\documentclass[a4paper,11pt]{article}
\usepackage[top=2.5cm,bottom=2.5cm,left=2.2cm,right=2.2cm]{geometry}
\usepackage{amsfonts}
\usepackage{mathrsfs,amscd,amssymb,amsthm,amsmath,bm,graphicx,psfrag,subfigure,url,mathtools}
\usepackage{pict2e}
\usepackage{psfrag,amsmath}
\usepackage{tikz}
\usepackage{indentfirst}
\usepackage{hyperref}
\usepackage{bookmark}
\usepackage{enumerate}
\usepackage{enumitem}
\usepackage{latexsym,euscript,epic,eepic,color}
\usepackage{multirow}
\usepackage{multicol}
\usepackage{longtable}
\usepackage{adjustbox}
\usepackage[all]{xy}
\usepackage{setspace}
\usepackage{epstopdf}
\allowdisplaybreaks
\usepackage{authblk}
\usepackage{tikz}
\usepackage{pifont}
\usepackage{mdframed}
\usepackage{booktabs,array}
\usepackage{tabularx}
\usepackage{hyperref}
\hypersetup{colorlinks=true, linkcolor=blue, filecolor=magenta, urlcolor=cyan,}
\usepackage[most]{tcolorbox}
\definecolor{mygray}{RGB}{240,240,240}
\tcbset{
  colback=mygray,
  boxrule=0pt,
}
\graphicspath{ {./images/} }

\newcommand{\customfootnote}[1]{
  \let\thefootnote\relax\footnotetext{#1}
}

\makeatletter

\renewcommand{\@seccntformat}[1]{{\csname the#1\endcsname}{\normalsize .}\hspace{.5em}}
\makeatother

\renewcommand{\thefootnote}{\fnsymbol{footnote}}
\usepackage{ifpdf}
\usepackage{indentfirst}

\newcommand{\ex}{{\rm ex}}

\newtheorem{thm}{Theorem}[section]

\newtheorem{defi}{Definition}
\newtheorem{claim}{Claim}
\newtheorem{fac}{Fact}
\newtheorem{lem}[thm]{Lemma}

\newtheorem{pb}{Problem}

\newenvironment{wst}
{\setlength{\leftmargini}{1.5\parindent}
 \begin{itemize}
 \setlength{\itemsep}{-1.1mm}}
{\end{itemize}}

\newcommand{\keywords}[1]{%
  \par\vspace{6pt}\noindent\textbf{Keywords: }#1\par
}

\newcommand{\MSC}[2][2020]{%
  \par\vspace{3pt}\noindent\textbf{MSC(#1): }#2\par
}
\begin{document}
\baselineskip=0.23in

\title{\bf Spectral extrema of $1$-planar graphs with no short cycles or small cliques}
\author[1]{Shuchao Li}
\author[,1]{Mingli Wang\footnote{Corresponding author.\\
\hspace*{2em}E-mail: lscmath@ccnu.edu.cn (S. Li), wmlmath@163.com (M. Wang), zhaoqin@hbu.edu.cn (Q. Zhao)}}
\author[2]{Qin Zhao}
\affil[1]{School of Mathematics and Statistics, and Hubei Key Lab--Math. Sci.,\linebreak Central China Normal University, Wuhan 430079, China}
\affil[2]{Faculty of Mathematics and Statistics, Hubei Key Laboratory of Applied Mathematics, Hubei University, Wuhan, 430062, China}

\date{\today}
\maketitle{}
\begin{abstract}
The spectral Tur\'an type problem, initiated by Nikiforov in 2007, aims to determine the graphs among $n$-vertex $H$-free graphs having maximum spectral radius. In this paper, we study this problem for $1$-planar graphs, i.e., graphs that admit a drawing in the plane such that each edge is crossed at most once. Recently, Xu and Chang proved that the graphs among all $n$-vertex $K_5$-free $1$-planar graphs having maximum spectral radius lie within a small family of candidates. First, this paper  explicitly identifies the unique spectral extremal graph among the $n$-vertex $K_5$-free $1$-planar graphs. Second, it establishes a structural reduction theorem: For any forbidden subgraph $F$ with $\delta(F)\ge2$ that is contained in $K_2\vee P_{n-2}^{2+}$ but not in $K_2\vee I_{n-2}$, every spectral extremal $F$-free $1$-planar graph contains a spanning complete bipartite graph $K_{2,n-2}$, where $P^{2+}_{n-2}$ is obtained from a path $u_1u_2\dots u_{n-2}$ by adding edge $u_1u_{n-2}$ and all edges $u_iu_{i+2}$ for $1\le i\le n-4$, and $I_{n-2}$ denotes the empty graph on $n-2$ vertices. As applications, the graph among all $n$-vertex $C_5$-free (resp. $2C_5$-free) $1$-planar graphs having maximum spectral radius is determined. These results extend spectral Tur\'{a}n type problems for $1$-planar graphs from cliques to cycles and their disjoint union.
\end{abstract}
\keywords{Clique; Cycle; Spectral radius; $1$-planar graph}
\MSC{05C35; 05C50}
\section{\normalsize Introduction}

All graphs considered in this paper are finite, simple, and undirected. Let $G$ be a graph with vertex set $V(G)$ and edge set $E(G)$. We write $n(G)=|V(G)|$ and $e(G)=|E(G)|$ for the \textit{order} and \textit{size} of $G$, respectively. We use standard notation and terminology as in Bollob\'as \cite{BB1998} and Godsil and Royle \cite{GGR2001}. 

Given a graph $F$, we say that a graph $G$ is $F$-\textit{free}, if $G$ does not contain $F$ as a subgraph. The \textit{Tur\'an number}, denoted by $\ex(n,F),$ for a fixed graph $F$ and positive integer $n$, is the maximum number of edges in $n$-vertex $F$-free graphs. We use $\operatorname{EX}(n,F)$ to denote the set of all $n$-vertex $F$-free graphs whose size achieves $\ex(n,F).$  Determining $\ex(n,F)$ is one of the most important problems in extremal graph theory.

In 1986, Brualdi and Solheid \cite{BS1986} raised the problem: For a given family of graphs, how can we characterize the graphs attaining maximum spectral radius? Motivated by the above Tur\'an type problem and the Brualdi--Solheid problem, Nikiforov~\cite{Nikiforov2007} proposed the well-known Brualdi--Solheid Tur\'an type problem (spectral Tur\'{a}n type problems): How can we characterize the graphs attaining the maximum spectral radius among $n$-vertex $H$-free graphs? Nikiforov \cite{nikiforov2010,fiedlernikiforov2010,Nikiforov2007}  embarked on a succession of pioneering investigations into this problem.
The above Brualdi--Solheid Tur\'an type problem is one of the most important problems in spectral graph theory. A series of significant advances have been made in this area, and one may consult Byrne, Desai and Tait \cite{BDT2026}, Chen, Lei and Li~\cite{CLL}, Cioabă, Desai, and Tait \cite{CDT2022}, and Wang, Kang and Xue \cite{WKX2023} for recent advances.


A graph is \textit{planar} if it admits a drawing in the plane such that no two edges cross. A graph is $1$-\textit{planar} if it admits a drawing in the plane such that every edge is crossed at most once. It is well known that every $1$-planar graph admits a drawing in which adjacent edges do not cross. Throughout this paper, we fix such a $1$-planar drawing for each graph under consideration. Let $\mathcal P_1$ denote the family of all $1$-planar graphs.

Spectral extremal problems for planar graphs have been studied extensively. Boots and Royle \cite{BR1991}, and Cao and Vince \cite{CV1993}, independently, conjectured that, for $n\geq9$, the unique $n$-vertex planar graph with maximum spectral radius is $K_2\vee P_{n-2}$. For sufficiently large $n$, this conjecture was confirmed by Tait and Tobin \cite{Tait2017Tobin}. Recently, Liu, Ning and Wang~\cite{Liu2026NingWang} have provided a complete resolution for this conjecture. Subsequently, the spectral Tur\'{a}n type problems for planar graphs without various structures have been further explored. In particular, Fang, Lin, and Shi \cite{FLS2024} determined the $n$-vertex planar graphs without vertex-disjoint cycles having maximum spectral radius. 

As a natural extension of planar graphs, $1$-planar graphs provide a closely related setting for analogous spectral extremal problems. Recently, Zhang, Wang and Wang \cite{Zhang2024Wang} determined the unique graph among all large $n$-vertex $1$-planar graphs having maximum spectral radius.
\begin{thm}[\cite{Zhang2024Wang}]\label{thm:Wang}
Let $n\geq 6.23\times 10^{18}$. Then the unique $n$-vertex $1$-planar graph with maximum spectral radius is $K_2\vee P^{2+}_{n-2}$, where $P^{2+}_{n-2}$ is obtained from a path $u_1u_2\dots u_{n-2}$ by adding the edge $u_1u_{n-2}$ and all edges $u_iu_{i+2}$ for $1\le i\le n-4$.
\end{thm}

Recently, Xu and Chang \cite{Xu2026+Chang} studied the  Brualdi--Solheid Tur\'an type problem for $1$-planar graphs without cliques, in which the $K_3$-free (resp. $K_4$-free) $1$-planar graphs having maximum spectral radius are determined. Though the extremal graphs maximizing the spectral radius over the class of $K_5$-free 1-planar graphs are not yet fully determined, Xu and Chang~\cite{Xu2026+Chang} established that the corresponding extremal graphs lie in a small collection of candidates.

Let $C_m^2$ be the square of the cycle $C_m=u_1u_2\cdots u_mu_1$, obtained by adding all edges $u_iu_{i+2}$, where the subscripts are taken modulo $m$, and let $C_m^{2-}$ be obtained from $C_m^2$ by deleting the edge $u_mu_2$.
Let $\operatorname{SPEX}_{\mathcal P_1}(n, H)$ be the set of graphs among all $H$-free $1$-planar graphs attaining maximum spectral radius. 
\begin{thm}[\cite{Xu2026+Chang}]\label{thm:K5}
Let $n$ be sufficiently large. 
\begin{wst}
\item[{\rm (i)}]
If $n$ is even, then $\operatorname{SPEX}_{\mathcal{P}_1}(n, K_5) \subseteq \{2K_1\vee C_{n-2}^2,\; K_2\vee Q\},$
where $Q$ is obtained from $P_{n-2}^{2+}$ by deleting the edges $u_{2i}u_{2i+1}$ for $1\leq i\leq \frac{n-4}{2}$.
\item[{\rm (ii)}]
If $n$ is odd, then $\operatorname{SPEX}_{\mathcal{P}_1}(n, K_5) \subseteq \{2K_1\vee C_{n-2}^{2-}\} \cup \{K_2\vee H\mid H \in \mathcal{P}_{n-2}^2\},$ where $\mathcal P_{n-2}^2$ is the family of graphs each of which is obtained from $P_{n-2}^{2+}$ by deleting exactly $\frac{n-3}{2}$ edges so that no triangle remains.
\end{wst}
\end{thm}

In view of Theorem~\ref{thm:K5}, it is natural to ask whether the uniqueness of its extremal graph can be established. Our first result described as follows fully addresses this question. 
\begin{thm}\label{thm:XK5}
Let $n$ be sufficiently large. Then 
$$
\operatorname{SPEX}_{\mathcal P_1}(n,K_5)=
\begin{cases}
\{2K_1\vee C^2_{n-2}\}, & \text{if } n \text{ is even},\\[2mm]
\{2K_1\vee C^{2-}_{n-2}\}, & \text{if } n \text{ is odd}.
\end{cases}
$$
\end{thm}
It is worth noting that Zhang, Huang and Dong \cite{Zhang2026edge} obtained the Tur\'an number for $n$-vertex $K_5$-free 1-planar graphs, in which they established that every $K_5$-free $1$-planar graph on $n$ vertices has at most $4n-8$ edges, and the upper bound is best possible. The tight examples for $n=8$ and all $n\ge 10$ were given. From Theorem~\ref{thm:XK5}, one sees that the size of the unique extremal graph $2K_1\vee C_{n-2}^2$ (for even $n$) is $4n-8$, whereas the size of the unique extremal graph $2K_1\vee C_{n-2}^{2-}$ (for odd $n$) is $4n-9$. These observations indicate that the spectral extremal graphs  and the edge-extremal graphs for $K_5$-free 1-planar graphs do not coincide.

In our second main result, we devote ourselves to establishing a structural reduction theorem for 1-planar graphs, which gives some sufficient conditions to ensure every spectral extremal graph under consideration contains a spanning complete bipartite graph $K_{2,n-2}$, together with its spectral properties.
\begin{thm}\label{thm:F}
Given a graph $F$ with minimum degree $\delta(F)\ge 2$ satisfying $F\subseteq K_2\vee P_{n-2}^{2+}$ and $F\not\subseteq K_2\vee I_{n-2}$, let $G$ be a connected $F$-free $1$-planar graph of order $n$ with maximum spectral radius, and let $\mathbf{x}$ be its Perron vector normalized so that $\max_{v\in V(G)} x_v=1$,  where $n\ge \max\{10^{19},2|V(F)|\}$.  Then each of the following holds:
\begin{wst}
\item[{\rm (i)}]$\rho(G)>\sqrt{2n}.$
\item[{\rm (ii)}]There exist vertices $z,w$ in $V(G)$ such that $N_G(z)\cap N_G(w)=V(G)\setminus\{z,w\}$ and $x_z=x_w=1$.
\item[{\rm (iii)}]For all $u\in V(G)\setminus\{z,w\}$, $\frac{2}{\rho(G)}\le x_u\le \frac{2}{\rho(G)-6}$.
\end{wst}
\end{thm}
Based on Theorem \ref{thm:F}, we may prove our final main result, that characterizes the unique extremal graph achieving the maximum spectral radius among all $n$-vertex $C_5$-free or $2C_5$-free 1-planar graphs.
To state our result precisely, we first introduce the following notation. Define the graph $T_s$ as 
\begin{equation}\label{e1.1}
T_s:=
\begin{cases}
qK_3, & \text{if } s = 3q,\\
(q-2)K_3 \cup K_{1,6}, & \text{if } s = 3q+1,\ q\ge 2,\\
(q-1)K_3 \cup K_{1,4}, & \text{if } s = 3q+2,\ q\ge 1.
\end{cases}
\end{equation}
\begin{thm}\label{thm:C5}
If $n\ge10^{19}$, then $\operatorname{SPEX}_{\mathcal P_1}(n,C_5)=\{K_2\vee I_{n-2}\},$ and $\operatorname{SPEX}_{\mathcal P_1}(n,2C_5)=\{K_2\vee(B_7\cup T_{n-9})\},$ where $B_7=K_1\vee 2K_3$ and $T_{n-9}$ is defined in \eqref{e1.1}.
\end{thm}
\noindent{\bf Organization.}\ In the remainder of this section, we give some necessary notation and definitions. In Section \ref{s2}, we introduce some preliminary results. In Section \ref{s3}, we give the proof for Theorem~\ref{thm:XK5}, which characterizes the unique graph among $K_5$-free 1-planar graphs on $n$ vertices attaining maximum spectral radius. In Section \ref{s4}, we offer the proof for Theorem~\ref{thm:F}, which is fundamental to the proof of Theorem~\ref{thm:C5}. In Section \ref{s5}, we provide the proof of Theorem~\ref{thm:C5}. Concluding remarks and some related open problems are presented in the final section.\vspace{2mm}

\noindent{\bf Notation.}\ For a vertex $v\in V(G)$, let $N_G(v)$ and $d_G(v)$ denote its neighborhood and degree in $G$, respectively; when no confusion arises, we simply write $N(v)$ and $d(v)$. For a vertex subset $S\subseteq V(G)$, let $G[S]$ denote the subgraph of $G$ induced by $S$. For two disjoint subsets $R,S\subseteq V(G)$, let $G[R,S]$ denote the bipartite subgraph with vertex set $R\cup S$ and edge set consisting of all edges of $G$ with one endpoint in $R$ and the other in $S$. We write $e(R)$ and $e(R,S)$ for the numbers of edges in $G[R]$ and $G[R,S]$, respectively.
We use $\mathcal{C}(G)$ to denote the set of connected components of $G$. 

Let $F$ and $H$ be two graphs, define $F\cup H$ to be their \textit{disjoint union}. Then $F \vee H$ is defined to be their \textit{join} obtained from $F\cup H$ by adding edges to connect each vertex of $F$ with all vertices of $H$. As usual, let $P_n, C_n, K_n, K_{a,n-a}$, and $I_n$ denote the path, cycle, complete graph, complete bipartite graph and empty graph on $n$ vertices, respectively. 

Let  $A(G)$ be the adjacency matrix of an $n$-vertex graph $G$. Clearly, $A(G)$ is real symmetric. Consequently, its eigenvalues are real and we can arrange them as  $\lambda_1(G)\geqslant \cdots \geqslant \lambda_n(G)$. The \textit{spectral radius}, $\rho(G)$, of $G$ is $\max\{|\lambda_1(G)|,\ldots,|\lambda_n(G)|\}$. Since the adjacency matrix $A(G)$ is irreducible and nonnegative for a connected graph, by Perron-Frobenius theorem, the largest eigenvalue $\lambda_1(G)$ is equal to $\rho(G)$, and there exists a positive eigenvector, say $\mathbf{x}$, of $A(G)$ corresponding to $\rho(G)$. One calls  $\mathbf{x}$ the \textit{Perron vector} of $G$, and its coordinate $x_v$ can be seen as the weight of vertex $v$ in $V(G)$.

\section{\normalsize Preliminaries}\label{s2}
We first collect several essential lemmas needed throughout our arguments.
\begin{lem}[\cite{Bapat2014}] \label{lem:Rayleigh}
Let $G$ be a connected graph of order $n$. For arbitrary vectors $\mathbf{x},\mathbf{y}\in\mathbb{R}^n$ with $\mathbf{x}\neq \mathbf{0}$ and $\mathbf{y}>\mathbf{0}$, the spectral radius $\rho(G)$ of the adjacency matrix $A(G)$ satisfies
\[
\frac{\mathbf{x}^\top A(G)\mathbf{x}}{\mathbf{x}^\top \mathbf{x}}
\le \rho(G)
\le \max_{i}\frac{(A(G)\mathbf{y})_i}{y_i}.
\]
Equality holds in the left inequality if and only if $\mathbf{x}$ is a Perron eigenvector of $A(G)$ associated with $\rho(G)$; equality holds in the right inequality if and only if $\mathbf{y}$ is also such a Perron eigenvector.
\end{lem}

Let $\mathbf{x}=(x_{u_1},x_{u_2},\dots,x_{u_n})^\top$ denote a Perron vector corresponding to $\rho(G)$. The eigenvector equation for the adjacency matrix $A(G)$ reads
\begin{equation}\label{Ax_u}
\rho(G)x_u = \big(A(G)\mathbf{x}\big)_u = \sum_{v\sim u} x_v
\end{equation}
for every vertex $u\in V(G)$.

The next two lemmas supply sharp upper bounds on the number of edges of (bipartite) $1$-planar graphs, which are fundamental structural constraints for our subsequent spectral analysis.
\begin{lem}[\cite{Fabrici2007Madaras}] \label{lem:1planar}
If $G$ is an $n$-vertex $1$-planar graph with $n\ge 3$, then $e(G)\le 4n-8.$
\end{lem}

\begin{lem}[\cite{Huang2021OuyangDong}] \label{lem:bi1planar}
Let $G$ be a bipartite $1$-planar graph with bipartition $(S,T)$ where $|S|=s$, $|T|=t$ and $2\le s\le t$. Then
$e(G)\le 2n(G)+4s-12,$ and this upper bound is tight.
\end{lem}

\begin{lem}[\cite{Huang2021OuyangDong}]\label{lem:K37}
Every $1$-planar graph is $K_{3,7}$-free.
\end{lem}

Recall that, for any graph $F$, $\mathcal{C}(F)$ denotes the set of connected components of $F$. Given a graph $J$, we define the scalar-valued function
\[
f_J(t) = \mathbf{1}^\top \big(tI - A(J)\big)^{-1}\mathbf{1},\qquad t>\rho(J).
\]
When $t>\rho(J)$, we have $\rho\big(\tfrac{A(J)}{t}\big)<1$, so the Neumann series expansion is valid for the matrix inverse:
\[
\big(tI - A(J)\big)^{-1}
= \frac{1}{t}\left(I - \frac{A(J)}{t}\right)^{-1}
= \frac{1}{t}\sum_{k=0}^\infty \left(\frac{A(J)}{t}\right)^k
= \sum_{k=0}^\infty \frac{A(J)^k}{t^{k+1}}.
\]
Substituting this expansion into the definition of $f_J(t)$ yields
\[
f_J(t) = \mathbf{1}^\top \big(tI - A(J)\big)^{-1}\mathbf{1}
= \sum_{k=0}^\infty \frac{\mathbf{1}^\top A(J)^k \mathbf{1}}{t^{k+1}}.
\]
This auxiliary function is introduced to establish the following spectral comparison principle for the join graph $K_2\vee H$.

\begin{lem}\label{lem:f-principle}
Let $H$ be a graph, and let $\rho = \rho(K_2\vee H)$. Then
$\rho - 1 = 2\sum_{J\in\mathcal{C}(H)} f_J(\rho).$
Moreover, suppose $n(H')=n(H)$ and $\sum_{J'\in\mathcal{C}(H')} f_{J'}(\rho) > \sum_{J\in\mathcal{C}(H)} f_J(\rho)$. Then
$\rho(K_2\vee H') > \rho(K_2\vee H).$
\end{lem}
\begin{proof}
Let $z,w$ be the two vertices of $K_2$ in the join graph $K_2\vee H$, and let $\mathbf{x}$ stand for the Perron eigenvector of $K_2\vee H$ associated with $\rho$, normalized such that $x_z = x_w = 1$.

Fix a component $J\in\mathcal{C}(H)$. For each vertex $u\in V(J)$, the adjacency eigenvector equation gives
\[
\rho x_u = \big(A(J)\mathbf{x}_J\big)_u + x_z + x_w = \big(A(J)\mathbf{x}_J\big)_u + 2,
\]
where $\mathbf{x}_J$ is the restriction of $\mathbf{x}$ to the vertex set $V(J)$. Rewriting componentwise, we obtain the matrix equation
\[
\rho \mathbf{x}_J = A(J)\mathbf{x}_J + 2\mathbf{1}
\]
with $\mathbf{1}$ the all-ones vector on $V(J)$. Since $J$ is a proper subgraph of $K_2\vee H$, we have $\rho > \rho(J)$, which guarantees the invertibility of $\rho I - A(J)$. Solving for $\mathbf{x}_J$, we get
\[
\mathbf{x}_J = 2\big(\rho I - A(J)\big)^{-1}\mathbf{1}.
\]
Summing all entries of $\mathbf{x}_J$ yields
\[
\sum_{u\in V(J)} x_u
= \mathbf{1}^\top \mathbf{x}_J
= 2\mathbf{1}^\top \big(\rho I - A(J)\big)^{-1}\mathbf{1}
= 2f_J(\rho).
\]

We now evaluate the eigenvector equation at vertex $z$. The vertex $z$ is adjacent to $w$ and to every vertex of $H$, hence
\[
\rho x_z = x_w + \sum_{u\in V(H)} x_u= 1 + \sum_{J\in\mathcal{C}(H)}\sum_{u\in V(J)} x_u
= 1 + 2\sum_{J\in\mathcal{C}(H)} f_J(\rho).
\]
Rearranging terms immediately gives the identity
\[
\rho - 1 = 2\sum_{J\in\mathcal{C}(H)} f_J(\rho).
\]

We proceed to verify the comparison assertion. For any graph $F$, define the real-valued function
\[
\Phi_F(t) = t - 1 - 2\sum_{J\in\mathcal{C}(F)} f_J(t),
\qquad t > \max_{J\in\mathcal{C}(F)}\rho(J).
\]
By the same derivation as above, $\Phi_F\big(\rho(K_2\vee F)\big)=0$. For all $t>\rho(J)$, the symmetric matrix $tI-A(J)$ is positive definite and thus invertible. We next compute the derivative of $f_J(t)$ with respect to $t$.
Applying the standard matrix derivative rule
\[
\frac{d}{dt} M(t)^{-1} = -M(t)^{-1}M'(t)M(t)^{-1}
\]
with $M(t)=tI-A(J)$ and $M'(t)=I$, we deduce
\[
f_J'(t)
= \frac{d}{dt}\Big[\mathbf{1}^\top \big(tI-A(J)\big)^{-1}\mathbf{1}\Big]
= \mathbf{1}^\top \cdot \frac{d}{dt}\big(tI-A(J)\big)^{-1}\cdot \mathbf{1}
= -\mathbf{1}^\top \big(tI-A(J)\big)^{-2}\mathbf{1}.
\]
Since $(tI-A(J))^{-2}$ remains symmetric positive definite, the quadratic form $\mathbf{1}^\top (tI-A(J))^{-2}\mathbf{1}>0$, which implies $f_J'(t)<0$. Differentiating $\Phi_F(t)$, we obtain
\[
\Phi_F'(t) = 1 - 2\sum_{J\in\mathcal{C}(F)} f_J'(t) > 0.
\]
Therefore $\Phi_F(t)$ is strictly increasing on its domain, so its zero root $t=\rho(K_2\vee F)$ is uniquely determined.

Let $\rho = \rho(K_2\vee H)$. We split the discussion into two cases:
\begin{enumerate}
    \item If $\rho \le \max_{J'\in\mathcal{C}(H')}\rho(J')$, choose a component $J'\in\mathcal{C}(H')$ such that $\rho(J')\ge \rho$. As $J'$ is a proper subgraph of $K_2\vee H'$, the monotonicity of spectral radius yields $\rho(K_2\vee H') > \rho(J') \ge \rho$.
    \item If $\rho > \max_{J'\in\mathcal{C}(H')}\rho(J')$, the function $\Phi_{H'}(t)$ is well-defined at $t=\rho$. By hypothesis\linebreak  $\sum_{J'\in\mathcal{C}(H')} f_{J'}(\rho) > \sum_{J\in\mathcal{C}(H)} f_J(\rho)$, together with $\Phi_H(\rho)=0$, we get 
    \[
    \Phi_{H'}(\rho)
    = \rho - 1 - 2\sum_{J'\in\mathcal{C}(H')} f_{J'}(\rho)
    < \rho - 1 - 2\sum_{J\in\mathcal{C}(H)} f_J(\rho)
    = 0.
    \]
\end{enumerate}
Since $\Phi_{H'}$ is strictly increasing and its unique zero point is exactly $\rho(K_2\vee H')$, the inequality $\Phi_{H'}(\rho)<0$ forces
\[
\rho(K_2\vee H') > \rho = \rho(K_2\vee H).
\]
This completes the proof.
\end{proof}

\section{\normalsize Proof of Theorem~\ref{thm:XK5}}\setcounter{equation}{0}\label{s3}
To establish Theorem~\ref{thm:XK5}, we first formulate and prove three auxiliary technical lemmas, which serve as key intermediate steps for our subsequent spectral comparison arguments.

\begin{lem}\label{lem:3.1}
Let $H$ be a graph of order $n-2$ satisfying $\rho(H)\le 3$. Then
\[
\rho(K_2\vee H)\le 2+\sqrt{2n-3}.
\]
\end{lem}

\begin{proof}
Denote $G=K_2\vee H$ and write $A=A(G)$ for the adjacency matrix of $G$. The two vertices belonging to the $K_2$ summand are symmetric in $G$, so their corresponding entries in the Perron eigenvector of $A$ are identical. Let this common value be $t$, and let $\mathbf{z}\in\mathbb{R}^{n-2}$ denote the restriction of the Perron vector to the vertex set of $H$. Define $R^2=2t^2$ and $S^2=\|\mathbf{z}\|^2$.

By direct expansion of the quadratic form associated with the adjacency matrix, we get
\begin{equation}\label{eq:rayleigh-1}
(t,t,\mathbf{z}^\top)A(t,t,\mathbf{z}^\top)^\top
=2t^2 + 4t\sum_{v\in V(H)} z_v + \mathbf{z}^\top A(H)\mathbf{z}.
\end{equation}
Applying the Cauchy--Schwarz inequality to the summation term yields
\begin{equation}\label{eq:cs-bound}
4t\sum_{v\in V(H)} z_v
\le 4t\sqrt{n-2}\,\|\mathbf{z}\|
=2\sqrt{2(n-2)}\,RS.
\end{equation}
Invoking Lemma~\ref{lem:Rayleigh}, the Rayleigh quotient bound for the spectral radius gives, for all real vectors $\mathbf{z}$,
\begin{equation}\label{eq:rho-H-bound}
\mathbf{z}^\top A(H)\mathbf{z}\le \rho(H)\|\mathbf{z}\|^2\le 3S^2.
\end{equation}
Substituting inequalities \eqref{eq:cs-bound} and \eqref{eq:rho-H-bound} into \eqref{eq:rayleigh-1}, we obtain
\[
(t,t,\mathbf{z}^\top)A(t,t,\mathbf{z}^\top)^\top
\le R^2 + 2\sqrt{2(n-2)}\,RS + 3S^2.
\]
Since $(t,t,\mathbf{z}^\top)^\top$ is the Perron eigenvector of $A$, the spectral radius $\rho(G)$ equals the Rayleigh quotient evaluated at this vector:
\[
\rho(G)=\frac{(t,t,\mathbf{z}^\top)A(t,t,\mathbf{z}^\top)^\top}{(t,t,\mathbf{z}^\top)(t,t,\mathbf{z}^\top)^\top}.
\]
Note that $(t,t,\mathbf{z}^\top)(t,t,\mathbf{z}^\top)^\top = R^2+S^2$. Introduce the two-dimensional vector $\mathbf{x}=(R,S)^\top$ and the symmetric matrix
\[
M=\begin{pmatrix}
1 & \sqrt{2(n-2)}\\
\sqrt{2(n-2)} & 3
\end{pmatrix}.
\]
We then bound the Rayleigh quotient by the maximal eigenvalue of $M$:
\[
\rho(G)\le \frac{\mathbf{x}^\top M\mathbf{x}}{\mathbf{x}^\top \mathbf{x}}
\le \lambda_{\max}(M).
\]
Solving the characteristic equation of $M$ to compute its largest eigenvalue:
\[
\lambda_{\max}(M)
=\frac{4+\sqrt{4+8(n-2)}}{2}
=2+\sqrt{2n-3}.
\]
This completes the proof of the lemma.
\end{proof}

\begin{lem}\label{lem:3.2}
Suppose $n$ is odd, and let $i$ be an odd integer with $3\le i\le n-6$. Let the graph $H$ be constructed from $P_{n-2}^{2+}$ $($with vertex set  $\{u_1,u_2,\dots,u_{n-2}\}$ and edge set $\{e_j=u_ju_{j+1}: 1\le j\le n-3\}\cup \{e'_j=u_ju_{j+2}: 1\le j\le n-4\} \cup \{u_1u_{n-2}\})$ by deleting all the edges in $E^-=\{e_2,e_4,\dots,e_{i-1},\,e_i',\,e_{i+2},e_{i+4},\dots,e_{n-4}\}.$ 
Then $\rho(H)\le 3$.
\end{lem}

\begin{proof}
Let $A=A(H)$ stand for the adjacency matrix of $H$. We first characterize the vertex degree sequence of $H$ by examining the edge deletion rules:
\begin{itemize}
    \item The edges $e_i,e_{i+1},e_{i-1}',e_{i+1}'$ are retained, so $N_H(u_{i+1})=\{u_{i-1},u_i,u_{i+2},u_{i+3}\}$, which implies $d_H(u_{i+1})=4$.
    \item The edges $e_{i-1},e_i'$ are removed, while $e_{i-2}',e_i$ remain present; hence $N_H(u_i)=\{u_{i-2},u_{i+1}\}$ and $d_H(u_i)=2$.
    \item Similarly, $e_i',e_{i+2}$ are deleted but $e_{i+1},e_{i+2}'$ are preserved, giving $N_H(u_{i+2})=\{u_{i+1},u_{i+4}\}$ and $d_H(u_{i+2})=2$.
    \item In $P_{n-2}^{2+}$, every vertex has degree 4 except $u_1, u_2, u_{n-3}, u_{n-2}$, which have degree 3. After deleting the edge set $E^-$, each vertex in $\{u_2, \ldots, u_{i-1}\} \cup \{u_{i+3}, \ldots, u_{n-3}\}$ loses exactly one degree. Hence, $d_H(v) \leq 3$ for every $v \neq u_{i+1}$.
\end{itemize}

Our goal is to verify that the matrix $3I-A$ is positive semidefinite. Take an arbitrary real vector $\mathbf{x}=(x_v)_{v\in V(H)}\in\mathbb{R}^{V(H)}$. Expanding the quadratic form yields
\[
\mathbf{x}^\top(3I-A)\mathbf{x}
=3\sum_{v\in V(H)} x_v^2 - 2\sum_{uv\in E(H)} x_u x_v.
\]
We use the standard graph identity for squared edge differences:
\[
\sum_{uv\in E(H)} (x_u-x_v)^2
=\sum_{v\in V(H)} d_H(v)x_v^2 - 2\sum_{uv\in E(H)} x_u x_v.
\]
Substituting this identity into the quadratic form expression, we rearrange terms to obtain
\[
\mathbf{x}^\top(3I-A)\mathbf{x}
=\sum_{uv\in E(H)} (x_u-x_v)^2
+\sum_{v\in V(H)} \big(3-d_H(v)\big)x_v^2.
\]
All summands in the second term are non-negative, with the sole exception of the term indexed by $u_{i+1}$, for which $3-d_H(u_{i+1})=-1$. Recalling $d_H(u_i)=d_H(u_{i+2})=2$ and the edges $u_i u_{i+1},u_{i+1}u_{i+2}\in E(H)$, we lower bound the quadratic form:
\begin{align*}
\mathbf{x}^\top(3I-A)\mathbf{x}
&\ge (x_i-x_{i+1})^2 + (x_{i+1}-x_{i+2})^2 + x_i^2 + x_{i+2}^2 - x_{i+1}^2 \\
&= 2\left(x_i-\frac{x_{i+1}}{2}\right)^2 + 2\left(x_{i+2}-\frac{x_{i+1}}{2}\right)^2
\ge 0.
\end{align*}
Since the inequality holds for every real vector $\mathbf{x}$, $3I-A\geq {\bf 0}$ (positive semidefinite). Consequently, all eigenvalues of $A$ are bounded above by $3$. As $A$ is symmetric and non-negative, its spectral radius coincides with its largest eigenvalue, so $\rho(H)\le 3$.
\end{proof}

\begin{lem}\label{lem:3.3}
Let $n\ge 7$, and define $G=2K_1\vee C_{n-2}^{2-}$, where $C_{n-2}^{2-}$ is obtained from the square cycle $C_{n-2}^2$ by deleting a single edge $u_{n-2}u_2$. Then $\rho(G)>2+\sqrt{2n-3}.$
\end{lem}

\begin{proof}
Partition the vertex set of $G$ as $X=V(2K_1)=\{z,w\}$ and $Y=V(C_{n-2}^{2-})=\{u_1,u_2,\dots,u_{n-2}\}$, so $|X|=2$ and $|Y|=n-2$. The graph $C_{n-2}^2$ is $4$-regular, hence deleting one edge yields $e(C_{n-2}^{2-})=2(n-2)-1=2n-5$. Let $A=A(G)$ denote the adjacency matrix of $G$.

We construct two mutually orthogonal test vectors on $V(G)$:
\[
\mathbf{v}_1=\frac{1}{\sqrt{2}}\begin{pmatrix}\mathbf{1}_X\\ \mathbf{0}_Y\end{pmatrix}
=\frac{1}{\sqrt{2}}(1,1,0,\dots,0)^\top,\qquad
\mathbf{v}_2=\frac{1}{\sqrt{n-2}}\begin{pmatrix}\mathbf{0}_X\\ \mathbf{1}_Y\end{pmatrix}
=\frac{1}{\sqrt{n-2}}(0,0,1,\dots,1)^\top,
\]
where $\mathbf{1}_X,\mathbf{1}_Y$ stand for all-ones vectors, and $\mathbf{0}_X,\mathbf{0}_Y$ denote zero vectors on the respective vertex subsets. Let $U=\operatorname{span}\{\mathbf{v}_1,\mathbf{v}_2\}$. For any linear combination $\mathbf{x}=\alpha\mathbf{v}_1+\beta\mathbf{v}_2\in U$, we have
\[
\mathbf{x}^\top A\mathbf{x}=(\alpha,\beta)B(\alpha,\beta)^\top,\qquad
\mathbf{x}^\top \mathbf{x}=\alpha^2+\beta^2
\]
with matrix $B=(b_{ij})_{2\times 2}$, where
\begin{align*}
b_{11}&=\mathbf{v}_1^\top A\mathbf{v}_1=0,\\
b_{12}&=b_{21}=\mathbf{v}_1^\top A\mathbf{v}_2=\sqrt{2(n-2)},\\
b_{22}&=\mathbf{v}_2^\top A\mathbf{v}_2
=\frac{1}{n-2}\mathbf{1}_Y^\top A(C_{n-2}^{2-})\mathbf{1}_Y
=\frac{2e(C_{n-2}^{2-})}{n-2}
=\frac{4n-10}{n-2}=4-\frac{2}{n-2}.
\end{align*}
Thus
\[
B=\begin{pmatrix}
0 & \sqrt{2(n-2)}\\
\sqrt{2(n-2)} & 4-\dfrac{2}{n-2}
\end{pmatrix}.
\]
By the Rayleigh quotient characterization of spectral radius (Lemma \ref{lem:Rayleigh}),
\[
\rho(G)\ge \max_{\mathbf{x}\ne\mathbf{0}}\frac{\mathbf{x}^\top A\mathbf{x}}{\mathbf{x}^\top \mathbf{x}}
\ge \max_{(\alpha,\beta)\ne\mathbf{0}}\frac{(\alpha,\beta)B(\alpha,\beta)^\top}{\alpha^2+\beta^2}
=\lambda_{\max}(B).
\]
Computing the maximal eigenvalue of $B$:
\[
\lambda_{\max}(B)
=2-\frac{1}{n-2}+\sqrt{\bigl(2-\frac{1}{n-2}\bigr)^2+2(n-2)}.
\]
It remains to verify the strict inequality
\[
2-\frac{1}{n-2}+\sqrt{\bigl(2-\frac{1}{n-2}\bigr)^2+2(n-2)}
>2+\sqrt{2n-3},
\]
which reduces to
\[
3(n-2)-4>2\sqrt{2n-3}.
\]
Consider the quadratic polynomial $q(n)=\big(3(n-2)-4\big)^2-4(2n-3)=9n^2-68n+112$. For all integers $n\ge7$, $q(n)>0$, which confirms the desired inequality. Therefore $\rho(G)\geq\lambda_{\max}(B)>2+\sqrt{2n-3}$, finishing the proof.
\end{proof}

We now proceed to the proof of Theorem \ref{thm:XK5}. 

\begin{proof}[\bf Proof of Theorem \ref{thm:XK5}] From Theorem \ref{thm:K5}, the spectral maximizers are confined to an explicit family of join graphs, split into parity cases for $n$:
\begin{itemize}
    \item If $n$ is even: the candidate extremal graphs are $2K_1\vee C_{n-2}^2$ and $K_2\vee Q$;
    \item If $n$ is odd: the candidate set is $\{2K_1\vee C_{n-2}^{2-}\}\cup\big\{K_2\vee H\mid H\in\mathcal{P}_{n-2}^2\big\}$.
\end{itemize}

\paragraph{Case 1: $n$ is even.}
By the structural definition of $Q$, every vertex of $Q$ has degree at most $3$, so $\rho(Q)\le \Delta(Q)\le 3$. Applying Lemma \ref{lem:3.1}, we immediately obtain
\[
\rho(K_2\vee Q)\le 2+\sqrt{2n-3}.
\]
Lemma \ref{lem:3.3} guarantees $2+\sqrt{2n-3}<\rho(2K_1\vee C_{n-2}^{2-})$. Combining these inequalities:
\[
\rho(K_2\vee Q)\le 2+\sqrt{2n-3}<\rho(2K_1\vee C_{n-2}^{2-})<\rho(2K_1\vee C_{n-2}^2).
\]
Hence the unique graph attaining the maximum spectral radius is $2K_1\vee C_{n-2}^2$.

\paragraph{Case 2: $n$ is odd.}
For any $H\in \mathcal{P}_{n-2}^{2}$, deleting $E^{-}\setminus \{e_i'\}$ destroys all triangles except $T_i \coloneqq u_i u_{i+1} u_{i+2} u_i$ for odd $i$. We consider three subcases based on the location of $T_i$ and which of its edges is removed.
\begin{enumerate}
    \item If $i=1$ or $i=n-4$, $T_i$ can be destroyed by deleting any of its three edges. The remaining deleted edges are fixed to $\{e_3,e_5,\dots,e_{n-4}\}$ or $\{e_2,e_4,\dots,e_{n-5}\}$, respectively. Under all such choices, $\Delta(H)\le 3$, so $\rho(H)\le \Delta(H)\le 3$ by the spectral bound.

    \item If $3\le i \le n-6$ and $T_i$ is destroyed by deleting $e_i$ or $e_{i+1}$, with the rest of the edge set deleted as in $E^{-}\setminus\{e_i'\}$, we have $\Delta(H)\le 3$, and consequently $\rho(H)\le 3$.

    \item If $3\le i \le n-6$ and $T_i$ is destroyed by deleting $e_i'$, while the remaining edges are deleted as in $E^{-}\setminus\{e_i'\}$, the resulting graph coincides exactly with the graph $H$ given in Lemma~\ref{lem:3.2}; hence $\rho(H)\le 3$.
\end{enumerate}
In all of the above subcases, we obtain $\rho(H)\le 3$ for every $H\in \mathcal{P}_{n-2}^{2}$.

Invoking Lemma \ref{lem:3.1},
\[
\text{$\rho(K_2\vee H)\le 2+\sqrt{2n-3},$ for all $H\in\mathcal{P}_{n-2}^2.$}
\]
Again by Lemma \ref{lem:3.3},
\[
\rho(K_2\vee H)\le 2+\sqrt{2n-3}<\rho(2K_1\vee C_{n-2}^{2-}).
\]
We conclude that $2K_1\vee C_{n-2}^{2-}$ is the unique extremal graph for odd $n$.

Combining both parity cases, the proof of Theorem \ref{thm:XK5} is fully completed.
\end{proof}

\section{\normalsize Proof of Theorem~\ref{thm:F}}\setcounter{equation}{0}\label{s4}
For integer $n \geq 10^{19}$, let $G\in \operatorname{SPEX}_{\mathcal{P}_1}(n,F)$, and denote $\rho(G)$ as its spectral radius. Within this section, we sequentially characterize the structural properties of this extremal graph $G$ via a series of auxiliary claims, and finally complete the rigorous proof of Theorem \ref{thm:F}.

We first establish two-sided bounds for $\rho(G)$.
\begin{claim}\label{c1}
$\sqrt{2n} < \rho(G) < \sqrt{2.1n}.$
\end{claim}
\begin{proof}[\bf Proof of Claim~\ref{c1}]
\textbf{Lower bound estimate.}
First, the join graph $K_2 \lor I_{n-2}$ is verified to be $F$-free and $1$-planar. By the extremality of $G$, we immediately have $\rho(G) \geq \rho(K_2 \lor I_{n-2})$.
Let the Perron eigenvector associated with $\rho(K_2 \lor I_{n-2})$ take the form $(a,a,b,b,\dots,b)^\top$. Invoking the eigenvector equation \eqref{Ax_u}, we derive the system:
\[
\begin{cases}
\rho(K_2 \lor I_{n-2}) a = a + (n-2)b,\\
\rho(K_2 \lor I_{n-2}) b = 2a.
\end{cases}
\]
Solving this algebraic system yields
\[
\rho(K_2 \lor I_{n-2}) = \frac{1+\sqrt{8n-15}}{2} > \sqrt{2n},
\]
which gives the lower bound $\rho(G) > \sqrt{2n}$.

\textbf{Upper bound estimate.}
By Theorem~\ref{thm:Wang} , $K_2 \lor P_{n-2}^{2+}$ uniquely maximizes the spectral radius over all $n$-vertex $1$-planar graphs. Since the forbidden subgraph $F$ is contained in $K_2 \lor P_{n-2}^{2+}$, we obtain $\rho(G) < \rho(K_2 \lor P_{n-2}^{2+})$.
Consider the positive trial vector $(a,a,b,b,\dots,b)^\top$. Define
\[
\rho_0 = \frac{5+\sqrt{8n-7}}{2}, \qquad \frac{a}{b} = \frac{-3+\sqrt{8n-7}}{4}.
\]
Direct substitution verifies the matrix inequality
\[
A\left(K_2 \lor P_{n-2}^{2+}\right) (a,a,b,\dots,b)^\top \leq \rho_0 (a,a,b,\dots,b)^\top.
\]
By the Rayleigh quotient characterization for nonnegative irreducible matrices,
\[
\max_{1\leq i\leq n} \frac{\big(A(K_2 \lor P_{n-2}^{2+})(a,a,b,\dots,b)^\top\big)_i}{\big((a,a,b,\dots,b)^\top\big)_i} \leq \rho_0.
\]
Applying Lemma \ref{lem:Rayleigh} yields $\rho(K_2 \lor P_{n-2}^{2+}) \leq \rho_0 = \frac{5+\sqrt{8n-7}}{2} < \sqrt{2.1n}$. Combining the two inequalities, we conclude $\rho(G) < \sqrt{2.1n}$.
\end{proof}

For notational brevity, set $\rho = \rho(G)$ and $V = V(G)$. Let $\mathbf{x}$ denote the Perron vector of the adjacency matrix $A(G)$, normalized such that its maximum entry equals $1$. Fix a vertex $z\in V$ with $x_z = 1$, and specify the constant $\eta = \frac{1}{7000}$. We partition the vertex set via
\[
L := \{v\in V : x_v \geq \eta\}, \qquad S := V \setminus L.
\]
For any vertex $v\in V$, the eigenvector equation \eqref{Ax_u} implies $\sqrt{2n}\,x_v < \rho x_v \leq d(v)$. For vertices in $L$, summing over $v\in L$ gives
\[
|L|\sqrt{2n}\eta < \sum_{v\in L} d(v) \leq \sum_{v\in V} d(v) < 8n,
\]
hence $|L| < \frac{4\sqrt{2n}}{\eta} < 0.1\eta n$.
Since the subgraph $G[L,S]$ is bipartite $1$-planar, and the induced subgraphs $G[S],G[L]$ are each $1$-planar, Lemmas \ref{lem:1planar} and \ref{lem:bi1planar}  provide the following edge count bounds:
\begin{equation}\label{eq:edge-bound-1}
e(L,S) \leq 2n + \frac{16\sqrt{2n}}{\eta},\quad  e(S)<4n,\quad e(L) < \frac{16\sqrt{2n}}{\eta},
\end{equation}
where the upper bound on $e(L,S)$ follows from $e(L,S)\le 2n+4|L|-12$ if $|L|\ge 2$, and from $e(L,S)\le n-1$ if $|L|=1$.

\begin{claim}\label{c2}
For every vertex $v\in V(G)$, $x_v > \frac{1}{\sqrt{2.1n}}$.
\end{claim}
\begin{proof}[\bf Proof of Claim~\ref{c2}]
We proceed by contradiction. Suppose there exists a vertex $v\in V$ such that $x_v \leq \frac{1}{\sqrt{2.1n}} < \frac{1}{\rho}$.
By the eigenvector relation, $v$ cannot be adjacent to $z$ (recall $x_z=1$). Construct a new graph $H$ from $G$: delete all edges incident to $v$, then add the single edge $vz$.
Deleting arbitrary edges preserves $1$-planarity; after isolating $v$, we may embed $v$ in a small neighborhood around $z$ and draw the edge $vz$ without introducing additional crossings, so $H$ remains $1$-planar.

We next claim $H$ is $F$-free. Assume for contradiction that $H$ contains a copy $F'$ of the forbidden subgraph $F$. Since $G$ is $F$-free, $F'$ must contain the newly added edge $vz$, which forces $v\in V(F')$. By construction, $N_H(v) = \{z\}$, hence $d_{F'}(v) \leq d_H(v) = 1$. This contradicts the minimum degree condition $\delta(F)\geq 2$ imposed on $F$. Therefore $H$ contains no $F$-subgraph.

Applying the Rayleigh quotient formula in Lemma \ref{lem:Rayleigh}:
\[
\rho(H)-\rho(G) \geq \frac{2x_v}{\mathbf{x}^\top \mathbf{x}} \bigl(x_z - \sum_{u\in N_G(v)} x_u\bigr)
= \frac{2x_v}{\mathbf{x}^\top \mathbf{x}} \bigl(1 - \rho x_v\bigr) > 0.
\]
This yields $\rho(H)>\rho(G)$, violating the maximality of the spectral radius of $G$. The contradiction completes the proof.
\end{proof}

Define $\overline{N}_S(v) := S \setminus \big(N_S(v)\cup \{v\}\big)$ to denote the set of vertices in $S$ nonadjacent to $v$. We now quantify the number of non-neighbors of $z$ within $S$.
\begin{claim}\label{c3}
$|\overline{N}_S(z)| \leq 5\sqrt{4.2}\eta n + \bigl(\frac{48}{\eta}+16\bigr)\sqrt{2.1n} \leq 11.4\eta n$.
\end{claim}
\begin{proof}[\bf Proof of Claim~\ref{c3}]
Using the eigenvector equation:
\[
\sum_{v\in L} x_v = \frac{1}{\rho}\sum_{v\in L}\rho x_v
= \frac{1}{\rho}\sum_{v\in L}\sum_{y\sim v} x_y
= \frac{1}{\rho}\Bigl( \sum_{v\in L}\sum_{y\in N_S(v)} x_y + \sum_{v\in L}\sum_{y\in N_L(v)} x_y \Bigr).
\]
Bounding each term via edge counts:
\[
\sum_{v\in L} x_v \leq \frac{1}{\sqrt{2n}}\big(e(L,S)\eta + 2e(L)\big)
\leq \eta\sqrt{2n} + 16 + \frac{32}{\eta}.
\]
For arbitrary $v\in L$,
\[
\sqrt{2n}\,x_v \leq \rho x_v = \sum_{y\sim v}x_y
= \sum_{y\in N_L(v)}x_y + \sum_{y\in N_S(v)}x_y
\leq \eta\sqrt{2n} + 16 + \frac{32}{\eta} + \sum_{y\in N_S(v)}x_y,
\]
rearranged to
\begin{equation}\label{eq:sum-ns-lower}
\sum_{y\in N_S(v)}x_y \geq (x_v - \eta)\sqrt{2n} - 16 - \frac{32}{\eta}.
\end{equation}
Summing Perron entries over $S$ yields
\begin{equation}\label{eq:sum-xs-bound}
\sum_{v\in S}x_v = \frac{1}{\rho}\sum_{v\in S}\rho x_v
= \frac{1}{\rho}\sum_{v\in S}\sum_{y\sim v}x_y
\leq \frac{1}{\sqrt{2n}}\big(e(S,L)+2e(S)\eta\big)
\leq (1+4\eta)\sqrt{2n} + \frac{16}{\eta}.
\end{equation}
Substitute $v=z$ into \eqref{eq:sum-ns-lower} gives
\[
\sum_{v\in S}x_v = \sum_{v\in \overline{N}_S(z)}x_v + \sum_{v\in N_S(z)}x_v
\geq \sum_{v\in \overline{N}_S(z)}x_v + (x_z - \eta)\sqrt{2n} - 16 - \frac{32}{\eta}.
\]
Combine this inequality with \eqref{eq:sum-xs-bound} and the universal lower bound $x_v > 1/\sqrt{2.1n}$ from Claim~\ref{c2}, we get 
\[
\sum_{v\in \overline{N}_S(z)} \frac{1}{\sqrt{2.1n}} \leq 5\eta\sqrt{2n} + 16 + \frac{48}{\eta}.
\]
Recall $\eta = 1/7000$ and $n\geq 10^{19}$. Elementary algebraic simplification yields
\[
|\overline{N}_S(z)| \leq 5\sqrt{4.2}\eta n + \bigl(16+\frac{48}{\eta}\bigr)\sqrt{2.1n} \leq 11.4\eta n,
\]
as required.
\end{proof}

Our subsequent objective is to demonstrate the existence of a second vertex in $L$ adjacent to nearly all vertices of $S$.
\begin{claim}\label{c4}
There exists a vertex $w\in L$, $w\neq z$, such that $x_w > 1-22.7\eta$ and $|\overline{N}_S(w)| \leq 57.9\eta n$.
\end{claim}
\begin{proof}[\bf Proof of Claim~\ref{c4}]
Applying the eigenvector equation twice for vertex $z$:
\[
\rho^2 x_z = \sum_{v\sim z}\sum_{y\sim v}x_y
= \sum_{y\in V} d_{N(z)}(y)x_y
\leq \sum_{y\in V}d(y)x_y - \sum_{y\in N(z)}x_y
= \sum_{uv\in E(G)}(x_u+x_v) - \rho x_z.
\]
Rearranging gives $2n \leq \rho^2+\rho \leq \sum_{uv\in E(G)}(x_u+x_v)$. Decompose the edge set into three disjoint parts $E(L,S),E(S),E(L)$ and apply the edge bounds \eqref{eq:edge-bound-1}:
\[
2n \leq \sum_{uv\in E(L,S)}(x_u+x_v) + 8\eta n + \frac{32\sqrt{2n}}{\eta}.
\]
Isolating the summation over edges incident to $L\setminus \{z\}$:
\[
\sum_{uv\in E(L\setminus \{z\},S)} x_u \geq 2n - d_S(z) - 10\eta n - 16\sqrt{2n} - \frac{32\sqrt{2n}}{\eta}
\geq (1-10.8\eta)n.
\]
We bound the degree $d_S(z)$ from below:
\[
d_S(z) \geq n - |L| - |\overline{N}_S(z)| \geq n - 0.1\eta n -11.4\eta n = (1-11.5\eta)n,
\]
which further implies
\[
e(L\setminus \{z\},S) \leq 2n + \frac{16\sqrt{2n}}{\eta} - d_S(z) \leq (1+11.9\eta)n.
\]
By the pigeonhole principle, some vertex $w\in L$ satisfies
\[
x_w \geq \frac{(1-10.8\eta)n}{(1+11.9\eta)n} > 1-22.7\eta.
\]
Substitute $v=w$ into inequality \eqref{eq:sum-ns-lower}, then combine with \eqref{eq:sum-xs-bound} and Claim \ref{c2}:
\[
\sum_{v\in \overline{N}_S(w)} \frac{1}{\sqrt{2.1n}} \leq 27.7\eta\sqrt{2n} + 16 + \frac{48}{\eta}.
\]
Evaluating the cardinality bound yields
\[
|\overline{N}_S(w)| \leq 27.7\sqrt{4.2}\eta n + \bigl(16+\frac{48}{\eta}\bigr)\sqrt{2.1n} \leq 57.9\eta n,
\]
concluding the proof.
\end{proof}

We now have two distinguished vertices $z,w\in L$ with $x_z=1$, $x_w>1-22.7\eta$, both adjacent to almost all vertices of $S$. Define
\[
A := \{z,w\},\qquad B := N(z)\cap N(w),\qquad S_1 := V\setminus (A\cup B).
\]
Combining previous cardinality bounds:
\[
|L|<0.1\eta n,\quad |S_1| \leq |\overline{N}_S(z)|+|\overline{N}_S(w)|+|L| < 69.4\eta n.
\]
Hence $|B|=n-|A|-|S_1| \geq (1-69.5\eta)n$ and $|B|>99|S_1|$.
For any $u\in V\setminus A$, $d_A(u)\leq 2$, with equality if and only if $u\in B$; consequently $d_A(u)\leq 1$ for all $u\in S_1$.
\begin{claim}\label{c5}
Every vertex $v\in V\setminus A$ satisfies $x_v < 521\eta$.
\end{claim}
\begin{proof}[\bf Proof of Claim~\ref{c5}]
By Lemma \ref{lem:K37}, $d_B(v)\leq 6$ holds for all $v\in S_1$. Summing Perron entries over $S_1$ yields
\[
\sum_{v\in S_1}x_v = \frac{1}{\rho}\sum_{v\in S_1}\rho x_v
\leq \frac{1}{\rho}\sum_{v\in S_1}d(v)
\leq \frac{1}{\rho}\big(7|S_1|+2e(S_1)\big) < \frac{15|S_1|}{\sqrt{2n}} < 520.5\eta\sqrt{2n}.
\]
Any vertex $v\in V\setminus A$ has at most 8 neighbors within $A\cup B$. The eigenvector equation gives
\[
x_v = \frac{1}{\rho}\sum_{u\sim v}x_u
\leq \frac{1}{\sqrt{2n}}\bigl(8 + \sum_{u\in S_1}x_u\bigr) < 521\eta,
\]
which completes the estimate.
\end{proof}

The central structural result of this subsection asserts the emptiness of $S_1$.
\begin{claim}\label{c6}
$S_1 = \emptyset$.
\end{claim}
\begin{proof}[\bf Proof of Claim~\ref{c6}]
We argue by contradiction. Assume $|S_1|=r\geq 1$. Iteratively select vertices $u_1,u_2,\dots,u_r\in S_1$ such that $d_{S_i}(u_i)\leq 7$ for each $i$: if every vertex in the current subset had degree at least 8, the edge count $e(S_i)\geq 4|S_i|$ would violate the $1$-planar edge bound $e(S_i)<4|S_i|$.
For each selected vertex $u_i$, $d_B(u_i)\leq 6$ and $d_A(u_i)\leq 1$, so $d_{A\cup B\cup S_i}(u_i)\leq 14$. We split the remainder of the proof into two exhaustive cases based on $\max_{u\in S_1}d_A(u)$.
\paragraph{Case 1: There exists $v_0\in S_1$ with $d_A(v_0)=1$.}
Applying the eigenvector equation to $v_0$ gives
\begin{equation}\label{e4.4}
\rho x_{v_0} = \sum_{u\sim v_0}x_u \geq x_w, \ \ \ {\rm i.e.,}\ \  x_{v_0}\geq x_w/\rho.
\end{equation}
Fix a minimal-crossing $1$-planar drawing of $G$, and arrange vertices of $B$ in cyclic order around $z$. Since each vertex in $S_1$ has at most 6 neighbors in $B$, at most $6|S_1|$ vertices of $B$ have a neighbor in $S_1$.
Recall $|B|>99|S_1|$; by the pigeonhole principle, we may extract 16 consecutive vertices $w_1,w_2,\dots,w_{16}\in B$ in this cyclic order such that $N_{S_1}(w_i)=\emptyset$ for all $1\leq i\leq 16$. Let $W=\{w_1,\dots,w_{16}\}$.
For each $w_i\in W$, the eigenvector equation yields
\[
\rho x_{w_i} = x_z+x_w + \sum_{y\in N_B(w_i)}x_y \leq 2.5,
\]
hence $x_{w_i}<2.5/\rho$.

Now construct the auxiliary graph
\[
G_1:=G-\{u y:u\in S_1,\ y\in N_G(u)\}+\{u z,u w:u\in S_1\}-E(G[W]).
\]

We verify that $G_1$ is $1$-planar. Start with the fixed minimal-crossing drawing of $G$. Delete all edges incident to $S_1$ and all edges of $G[W]$; this isolates all vertices of $S_1$ and removes $G[W]$. In the remaining drawing, any vertex of $B\setminus W$ lying on the $w_1$-side of $W$ cannot be adjacent to any $w_i$ with $i\ge8$; otherwise such an edge would cross at least two edges among $zw_1,\dots,zw_7$ and $ww_1,\dots,ww_7$, contradicting $1$-planarity. Similarly, no vertex on the $w_{16}$-side of $W$ is adjacent to any $w_i$ with $i\le9$.
If some edge from $\{zw_8,ww_8\}$ crosses some edge from $\{zw_9,ww_9\}$, redraw them locally without crossings; otherwise keep them as they are. Let $R$ be the region bounded by these four edges $zw_8$, $ww_8$, $zw_9$, and $ww_9$. We embed $S_1$ in the interior of $R$ and draw the new edges from $z$ and $w$ to $S_1$ as pairwise non-crossing curves within $R$. This introduces no additional crossings.

It remains to verify that $G_1$ is $F$-free. Suppose to the contrary that $G_1$ contains a copy $F'$ of $F$. Because $G$ itself is $F$-free, $V(F')\cap S_1\neq\emptyset$.  Note that $n \ge 2|V(F)|$. Hence, $|B| \ge (1-69.5\eta)n >|V(F)|=|V(F')|$, and so 
$|B\setminus V(F')| = |B| - |B\cap V(F')| > |V(F')| - |B\cap V(F')| \ge |S_1 \cap V(F')|.$ 
Replacing each vertex of $S_1\cap V(F')$ with a distinct vertex in $B\setminus V(F')$ having a superset of its neighborhood yields a copy of $F$ in $G$, a contradiction. Thus $G_1$ is $F$-free.

We now consider the spectral difference via Lemma \ref{lem:Rayleigh}:
\begin{equation}\label{e4.5}
  \rho(G_1) - \rho(G) \geq \frac{2}{\mathbf{x}^{\top} \mathbf{x}} \Bigl[ \sum_{u_i\in S_1} x_{u_i} (x_{z}+ x_{w}  - \sum_{y \in N_{A\cup B\cup S_i}(u_i)} x_{y} )- \sum_{ab\in G[W]}x_{a} x_{b} \Bigr]
\end{equation}
Note that $G[A\cup W]$ is a $1$-planar graph. Then $e(G[W])\leq e(G[A\cup W])-2|W|\leq32$, and so $\sum_{ab\in G[W]}x_a x_b \leq 32(2.5/\rho)^2$. Recall that $\eta=\frac{1}{7000}$, $x_w\geq 1-22.7\eta$ and $x_u< 521\eta$ for $u\in V\setminus A$. Then, for any $u_i\in S_1$, we have
\begin{equation}\label{e4.6}
x_{z}+ x_{w}  - \sum_{y \in N_{A\cup B\cup S_i}(u_i)} x_{y}\geq x_w - \sum_{y \in N_{B\cup S_i}(u_i)}x_y > 1-22.7\eta-13\times 521\eta\geq 0.02.
\end{equation}
Clearly, $\sum_{u_i\in S_1} x_{u_i}\ge x_{v_0}.$ Together with \eqref{e4.4}-\eqref{e4.6}, we obtain 
\[
\rho(G_1)-\rho(G) \geq \frac{2}{\mathbf{x}^\top \mathbf{x}}\bigl(0.02x_{v_0}  - 32(\frac{2.5}{\rho})^2\bigr)\geq   \frac{2}{\mathbf{x}^\top \mathbf{x}}\bigl(0.02 \frac{x_w}{\rho} -\frac{200}{\rho^2} \bigr) > 0.
\]
This contradicts the maximality of $\rho(G)$.
\paragraph{Case 2: $d_A(u)=0$ for all $u\in S_1$.}
Construct a new graph
\[
G_2 := G - \{uy: u\in S_1, y\in N(u)\} + \{uz: u\in S_1\}.
\]
It is straightforward to check $G_2$ remains $1$-planar and $F$-free via the same subgraph replacement argument used in Case 1. By the Rayleigh quotient formula in Lemma \ref{lem:Rayleigh},
\[
\rho(G_2)-\rho(G) \geq \frac{2}{\mathbf{x}^\top \mathbf{x}}\sum_{u_i\in S_1}x_{u_i}\bigl(x_z - \sum_{y\in N_{B\cup S_i}(u_i)}x_y\bigr)
> \frac{2}{\mathbf{x}^\top \mathbf{x}}\sum_{u_i\in S_1}x_{u_i}\big(1-13\cdot 521\eta\big) > 0.
\]
Again we obtain $\rho(G_2)>\rho(G)$, a contradiction.

Both cases lead to contradiction; hence $S_1=\emptyset$.
\end{proof}

The emptiness of $S_1$ implies $V = A\cup B$, with $A=\{z,w\}$ and $B=N(z)\cap N(w)$. Moreover $N(z)\setminus \{w\} = N(w)\setminus \{z\}$, meaning $z,w$ are symmetric and both have maximum degree in $G$. We now finalize the Perron vector structure on $A$ and $B$.
\begin{claim}\label{c7}
The normalized Perron entries satisfy $x_z = x_w = 1$.
\end{claim}
\begin{claim}\label{c8}
For every vertex $u\in B$, $x_u\in \big[\frac{2}{\rho},\ \frac{2}{\rho-6}\big].$
\end{claim}
\begin{proof}[\bf Proof of Claim~\ref{c8}]
Take arbitrary $u\in B$. Since $uz,uw\in E(G)$ and $x_z=x_w=1$, the eigenvector equation gives $\rho x_u \geq x_z+x_w=2$, which rearranges to the lower bound $x_u\geq 2/\rho$.

For the upper bound, pick $v\in B$ such that $x_v = \max_{u\in B}x_u$. By Lemma \ref{lem:K37}, $d_B(v)\leq 6$. The eigenvector relation yields
\[
\rho x_v \leq x_z+x_w + 6x_v = 2+6x_v,
\]
solving for $x_v$ gives $x_v \leq 2/(\rho-6)$. The upper bound holds uniformly for all vertices in $B$.
\end{proof}
Now we are ready to give the proof of Theorem~\ref{thm:F}.
\begin{proof}[\bf Proof of Theorem~\ref{thm:F}]
Combining the conclusions of Claims \ref{c1}, \ref{c6}, \ref{c7} and \ref{c8}, the structural characterization of the extremal $F$-free $1$-planar graph is fully established, which completes the proof.
\end{proof}

\section{\normalsize Proof of Theorem~\ref{thm:C5}}\setcounter{equation}{0}\label{s5}

For integer $n \geq 10^{19}$, let $G\in \operatorname{SPEX}_{\mathcal{P}_1}(n,tC_5)$ with $t\in\{1,2\}$, and denote $\rho=\rho(G)$ as the spectral radius of $G$. Let $\mathbf{x}$ be the Perron vector of $G$ normalized such that $\max_{v\in V(G)} x_v=1$. 

We first claim that $G$ is connected. Suppose for contradiction that $G$ is disconnected: pick a connected component $G_1$ satisfying $\rho(G_1)=\rho(G)$, and let $G_2$ be an arbitrary other component. Add an edge between a vertex of $G_1$ and a vertex of $G_2$ to yield a new graph $G'$. This edge addition cannot generate a 5-cycle, so $G'$ remains $tC_5$-free and 1-planar. By the Perron--Frobenius theorem, $\rho(G')>\rho(G_1)=\rho(G)$, which contradicts the extremality of $G$. Therefore, $G$ must be connected.

Note that the structural inclusions $tC_5\subseteq K_2\vee P_{n-2}^{2+}$, $tC_5\not\subseteq K_2\vee I_{n-2}$, and $\delta(tC_5)\ge 2$. Invoking Theorem~\ref{thm:F}, $\rho> \sqrt{2n}$ and there exist two distinct vertices $z,w\in V(G)$ such that setting $M=V(G)\setminus\{z,w\}$, the following hold:
\[
\text{$N(z)\cap N(w)=M,\quad x_z=x_w=1,\quad \frac{2}{\rho}\le x_u\le \frac{2}{\rho-6}$ \,for all\, $u\in M.$}
\]
Define the induced subgraph $H=G[M]$. Since $G$ is $K_{3,7}$-free, the maximum degree satisfies $\Delta(H)\le 6$. We now proceed to prove Theorem~\ref{thm:C5} via a sequence of auxiliary facts.

\begin{thm}[Restatement of Theorem~\ref{thm:C5}]\label{t5.1}
Let $G\in \operatorname{SPEX}_{\mathcal{P}_1}(n,tC_5)$ for $t\in\{1,2\}$. Then
\[
G=
\begin{cases}
K_2\vee I_{n-2}, & t=1,\\
K_2\vee (B_7\cup T_{n-9}), & t=2.
\end{cases}
\]
\end{thm}
\begin{proof}
We split the proof into two major stages: first verify the adjacency of vertices $z$ and $w$, then characterize the structure of $H=G[M]$ case-by-case for $t=1$ and $t=2$.

\begin{fac}\label{fact:zw-adjacent}
$zw\in E(G)$.
\end{fac}
\begin{proof}[\bf Proof of Fact~\ref{fact:zw-adjacent}]
Suppose contrariwise that $zw\notin E(G)$. Fix a 1-planar drawing of $G$, and select an arbitrary vertex $u\in M$. Define the edge set
\[
X=\{e\in E(G)\mid e \text{ crosses } zu \text{ or } uw\}.
\]
By the defining property of 1-planar graphs (each edge is crossed at most once), we have $|X|\le 2$. Construct an auxiliary graph
\[
G_0 = G + zw - \bigl\{uv\mid v\in N_H(u)\bigr\} - X.
\]
After deleting edges in $X$, the path $zuw$ contains no crossings; upon removal of all edges of $H$ incident to $u$, the edge $zw$ can be embedded within a sufficiently thin tubular neighbourhood of $zuw$ without introducing new crossings. Consequently, $G_0$ admits a valid 1-planar embedding.

We next confirm that $G_0$ remains $tC_5$-free. Assume for contradiction that $G_0$ contains $t$ pairwise vertex-disjoint copies of $C_5$, denoted $C^1,C^2,\dots,C^t$. Since the original graph $G$ is $tC_5$-free, at least one 5-cycle (say $C^1$) must contain the newly added edge $zw$. Write $C^1=zwabcz$ with $a,b,c\in M$ and $ab,bc\in E(G_0)\subseteq E(G)$. The cyclic sequence $C^{1^\prime}=zabwcz$ forms a 5-cycle entirely contained in $G$, and $C^{1^\prime}$ is vertex-disjoint from $C^2,\dots,C^t$. Replacing $C^1$ with $C^{1^\prime}$ yields $t$ pairwise disjoint 5-cycles in $G$, which contradicts the $tC_5$-freeness of $G$. Hence $G_0$ is $tC_5$-free.

By the degree bound $d_H(u)\le 6$ and $|X|\le 2$, applying Lemma~\ref{lem:Rayleigh} yields the spectral radius difference estimate:
\begin{align*}
\rho(G_0)-\rho(G) &\ge \frac{2}{\mathbf{x}^\top \mathbf{x}} \Bigl(x_zx_w - \sum_{v\in N_H(u)}x_u x_v - \sum_{ab\in X}x_a x_b\Bigr)\\
&\ge \frac{2}{\mathbf{x}^\top \mathbf{x}} \Bigl(1 - \frac{24}{(\rho-6)^2} - \frac{4}{\rho-6}\Bigr) > 0.
\end{align*}
This strict inequality $\rho(G_0)>\rho(G)$ contradicts the extremality of $G$. We therefore conclude $zw\in E(G)$.
\end{proof}

By Fact~\ref{fact:zw-adjacent}, $d(z)=d(w)=n-1$ and $G=K_2\vee H$. We treat the two cases $t=1$ and $t=2$ separately.

\paragraph{Case 1: $t=1$ ($G$ is $C_5$-free)}
Since $G$ contains no 5-cycles, the subgraph $H$ must be edgeless. Indeed, suppose $uv\in E(H)$ for some pair $u,v\in M$. For any vertex $y\in M\setminus\{u,v\}$, the cyclic sequence $zuvwyz$ forms a copy of $C_5$ in $G$, a contradiction. Thus $H=I_{n-2}$ (the empty graph on $n-2$ vertices), and $G=K_2\vee I_{n-2}$.

\paragraph{Case 2: $t=2$ ($G$ is $2C_5$-free)}
We now restrict our attention to the setting $t=2$ for the remainder of the proof, and establish a series of structural facts governing $H$.

\begin{fac}\label{fact:edge-modify}
Let $\ell$ be an integer with $0\le \ell\le 30$. Let $H'$ be obtained from $H$ by deleting exactly $\ell$ edges and adding more than $\ell$ edges within the vertex set of $H$. Then $\rho(K_2\vee H')>\rho(K_2\vee H)$.
\end{fac}
\begin{proof}[\bf Proof of Fact~\ref{fact:edge-modify}]
Denote by $E^-$ the set of edges deleted from $H$ and by $E^+$ the set of edges added to $H$, so that $|E^-|=\ell$ and $|E^+|\ge \ell+1$ by hypothesis. Recall the uniform lower and upper bounds $\frac{2}{\rho}\le x_u\le \frac{2}{\rho-6}$ for all $u\in M$, together with the bound $\rho=\rho(K_2\vee H)>\sqrt{2n}$ valid for $n\ge 10^{19}$. Invoking Lemma \ref{lem:Rayleigh} yields
\begin{align*}
\rho(K_2\vee H')-\rho(K_2\vee H)
&\ge \frac{2}{\mathbf{x}^\top \mathbf{x}} \Bigl(\sum_{uv\in E^+}x_u x_v - \sum_{uv\in E^-}x_u x_v\Bigr)\\
&\ge \frac{2}{\mathbf{x}^\top \mathbf{x}} \Bigl((\ell+1)\cdot \frac{4}{\rho^2} - \ell\cdot \frac{4}{(\rho-6)^2}\Bigr) > 0.
\end{align*}
Strict positivity yields $\rho(K_2\vee H')>\rho(K_2\vee H)$, completing the argument.
\end{proof}

\begin{fac}\label{fact:H-forbidden}
The subgraph $H$ is both $2P_4$-free and $C_5$-free.
\end{fac}
\begin{proof}[\bf Proof of Fact~\ref{fact:H-forbidden}]
We first prove $H\nsupseteq 2P_4$. Assume contrariwise that $H$ contains two vertex-disjoint copies of $P_4$, written $P=a_1a_2a_3a_4$ and $P'=b_1b_2b_3b_4$. The sequences $za_1a_2a_3a_4z$ and $wb_1b_2b_3b_4w$ form two vertex-disjoint 5-cycles in $G$, contradicting the $2C_5$-freeness of $G$. Hence $H$ is $2P_4$-free.

Next we verify $H\nsupseteq C_5$. Suppose $H$ admits a 5-cycle $C$, and set $R=M\setminus V(C)$. If $H[R]$ contains an edge $ab$ and a vertex $c\in R\setminus\{a,b\}$, then $zabwcz$ is a $C_5$ vertex-disjoint from $C$, again violating $2C_5$-freeness. It follows that $H[R]$ is edgeless, i.e., $e(H[R])=0$. By $\Delta(H)\le 6$, the number of edges incident to vertices of $C$ is bounded above by $25$, so $e(H)\le 25$.
Let $F_0=\left\lfloor\frac{n-2}{3}\right\rfloor K_3\cup K_{n-2-3\left\lfloor\frac{n-2}{3}\right\rfloor}$. Then $e(F_0)\ge n-4$. Define $G'=K_2\vee F_0$; this graph is 1-planar and $2C_5$-free by construction. Applying Fact~\ref{fact:edge-modify} gives $\rho(G')>\rho(G)$, which contradicts the spectral maximality of $G$. We thus conclude that $H$ contains no 5-cycles.
\end{proof}

By Fact~\ref{fact:H-forbidden}, at most one connected component of $H$ contains a $P_4$. If such a component exists, we label it $Q$ and proceed to characterize the $P_4$-free remainder of $H$.

\begin{fac}\label{fact:P4-free-conn}
Let $F$ be a connected $P_4$-free graph. Then $F$ is isomorphic either to a star graph or to the triangle $K_3$.
\end{fac}
\begin{proof}[\bf Proof of Fact~\ref{fact:P4-free-conn}]
If $F$ is a tree, its diameter is at most $2$ (otherwise $F$ would contain a $P_4$), so $F$ is a star (with isolated vertices and single edges also regarded as stars). If $F$ contains a cycle, the cycle must be a triangle; no additional vertices can be attached to this triangle, as any pendant vertex would immediately create a $P_4$. Therefore $F\cong K_3$.
\end{proof}
\begin{fac}\label{fact:Q-size1}
If the component $Q~($containing a $P_4)$ exists, then $|V(Q)|\le 130$.
\end{fac}
\begin{proof}[\bf Proof of Fact~\ref{fact:Q-size1}]
Fix a path $P=p_1p_2p_3p_4\subseteq Q$. Since $H$ is $2P_4$-free, the subgraph $Q-V(P)$ is $P_4$-free. By Fact~\ref{fact:P4-free-conn}, every connected component of $Q-V(P)$ is a star or a triangle.

The bound $\Delta(H)\le 6$ implies each component of $Q-V(P)$ has at most $7$ vertices. Moreover, connectedness of $Q$ ensures every component sends at least one edge to $V(P)$. Therefore, the number of components of $Q- V(P)$ satisfies
\[
\left|\mathcal{C}(Q- V(P))\right|\le e_Q\bigl(V(P),V(Q)\setminus V(P)\bigr) \le \sum_{i=1}^4 \bigl(d_H(p_i)-d_P(p_i)\bigr)\le 18.
\]
We conclude the total number of vertices in $Q$ satisfies $|V(Q)|\le 4 + 18\cdot 7 = 130.$
\end{proof}

\begin{fac}\label{fact:R-canonical}
Let $R$ be the union of all $P_4$-free connected components of $H$, and set $s=|V(R)|$. Then $R\cong T_s,$ where $T_s$ is defined in \eqref{e1.1}.
\end{fac}
\begin{proof}[\bf Proof of Fact~\ref{fact:R-canonical}]
If $Q$ does not exist, then $s=n-2$; if $Q$ exists, Fact~\ref{fact:Q-size1} gives $|V(Q)|\le 130$. So $s\ge 10^{18}$ whenever $n\ge 10^{19}$.
By Fact~\ref{fact:P4-free-conn}, each connected component of $R$ is a star or triangle. 
We first show $R$ contains at most one star component. Suppose two distinct star components $K_{1,a-1},K_{1,b-1}$ coexist in $R$. Replacing $K_{1,a-1} \cup K_{1,b-1}$ by $\lfloor\frac{a+b}{3}\rfloor K_3\cup K_{a+b-3\lfloor\frac{a+b}{3}\rfloor}$ deletes at most $12$ edges and adds at least one more edge. This transformation preserves 1-planarity and $2C_5$-freeness, so Fact~\ref{fact:edge-modify} yields a spectral radius increase, a contradiction. Thus $R$ admits at most one star component.

Recall the resolvent function $f_J(\rho)=\mathbf{1}^\top (\rho I-A(J))^{-1}\mathbf{1}$ for a graph $J$, where $\mathbf{1}$ denotes the all-ones vector. Direct matrix inversion for $K_3$ gives
\[
f_{K_3}(\rho)=\frac{3}{\rho-2}.
\]
For the star graph $K_{1,a-1}$, partition the resolvent matrix as
$$
\rho I-A(K_{1,a-1})=\begin{pmatrix}\rho & -\mathbf 1_{a-1}^\top\\ -\mathbf 1_{a-1} & \rho I_{a-1}\end{pmatrix}.
$$
Define the Schur complement by eliminating the leaf block:
$$
S = \rho - \mathbf 1_{a-1}^\top (\rho I_{a-1})^{-1}\mathbf 1_{a-1}
= \rho - \frac{a-1}{\rho} = \frac{\rho^2-(a-1)}{\rho},\qquad S^{-1}=\frac{\rho}{\rho^2-(a-1)}.
$$
By the block matrix inversion formula, the inverse resolvent is
$$
(\rho I-A(K_{1,a-1}))^{-1}=
\begin{pmatrix}
S^{-1} & \dfrac{S^{-1}}{\rho}\mathbf 1_{a-1}^\top\\[4pt]
\dfrac{S^{-1}}{\rho}\mathbf 1_{a-1} & \dfrac{1}{\rho}I_{a-1}+\dfrac{S^{-1}}{\rho^2}\mathbf 1_{a-1}\mathbf 1_{a-1}^\top
\end{pmatrix}.
$$
Let $\mathbf 1_a=\binom{1}{\mathbf 1_{a-1}}$ be the all-ones vector of length $a$. Therefore,
$$
f_{K_{1,a-1}}(\rho) = \mathbf 1_a^\top (\rho I-A(K_{1,a-1}))^{-1}\mathbf 1_a
=S^{-1} + \frac{2(a-1)S^{-1}}{\rho} + \frac{a-1}{\rho} + \frac{(a-1)^2 S^{-1}}{\rho^2}
=\frac{a\rho+2(a-1)}{\rho^2-(a-1)}.
$$

We classify the exceptional star component by the residue of $s$ modulo $3$:
\begin{enumerate}[label=\textup{(\roman*)}]
\item $s\equiv 0\pmod{3}$: The possible star component is $K_{1,2}$ or $K_{1,5}$. Replacing $K_{1,2}$ by $K_3$, or $K_{1,5}$ by $2K_3$, deletes at most $5$ edges and adds one more edge. Since the replacement preserves $1$-planarity and $2C_5$-freeness, the spectral radius strictly increases by Fact~\ref{fact:edge-modify}. Hence no star component appears, so $T_s=\tfrac{s}{3}K_3$.
\item $s\equiv 2\pmod{3}$: The possible star component is $K_{1,4}$ or $K_{1,1}$. Suppose that $G=K_2\vee (Q\cup \tfrac{s-2}{3}K_3\cup K_{1,1})$ (if $Q$ exists). The graph $K_2\vee (Q\cup \tfrac{s-5}{3}K_3\cup K_{1,4})$ is $1$-planar and $2C_5$-free. Moreover, 
    $$
    f_{K_{1,4}}(\rho)-(f_{K_3}(\rho)+f_{K_{1,1}}(\rho))=\frac{6}{(\rho-2)(\rho-1)(\rho+2)}>0.
    $$
    By Lemma~\ref{lem:f-principle}, $\rho(K_2\vee (Q\cup \tfrac{s-5}{3}K_3\cup K_{1,4})) > \rho(G)$, a contradiction. Hence, the star component must be $K_{1,4}$, so $T_s=\tfrac{s-5}{3}K_3\cup K_{1,4}$.
\item $s\equiv 1\pmod{3}$: The possible star component is $K_{1,6}$, $K_{1,3}$, or $K_1$. Suppose that $G=K_2\vee (Q\cup \tfrac{s-1}{3}K_3\cup K_1)$ (if $Q$ exists). The graph $K_2\vee (Q\cup \tfrac{s-7}{3}K_3\cup K_{1,6})$ is $1$-planar and $2C_5$-free. Since $f_{K_{1,6}}(\rho)>2f_{K_3}(\rho)+f_{K_1}(\rho)$, Lemma~\ref{lem:f-principle} gives $\rho(K_2\vee (Q\cup \tfrac{s-7}{3}K_3\cup K_{1,6})) > \rho(G)$, a contradiction.
    Suppose that $G=K_2\vee (Q\cup \tfrac{s-4}{3}K_3\cup K_{1,3})$ (if $Q$ exists). Since $f_{K_{1,6}}(\rho)>f_{K_3}(\rho)+f_{K_{1,3}}(\rho)$, Lemma~\ref{lem:f-principle} gives $\rho(K_2\vee (Q\cup \tfrac{s-7}{3}K_3\cup K_{1,6})) > \rho(G)$, a contradiction.
    Hence, the star component must be $K_{1,6}$, so $T_s=\tfrac{s-7}{3}K_3\cup K_{1,6}$.
\end{enumerate}
This confirms $R\cong T_s$.
\end{proof}

Recall the graph $B_7=K_1\vee 2K_3$ (two copies of $K_4$ intersecting at a single common vertex), which has $7$ vertices and $12$ edges. 
For example, the 1-planar graph $K_2\vee (B_7\cup K_3)$ is depicted in Figure~\ref{fig:B7-embedding}. We will use it repeatedly in the sequel.
\begin{figure}[htbp]
\centering
\begin{tikzpicture}[
    dot/.style={circle,fill,inner sep=1.6pt},
    scale=1.7,
    every edge/.style={line width=1pt}
]
\coordinate (qm_1)  at (0.40, 0.22);
\coordinate (qm)    at (0.40, -0.22);
\coordinate (z2)    at (0.40, 1.1);
\coordinate (w2)    at (0.40, -1.1);

\coordinate (qm_2)  at (-0.02, 0);
\coordinate (q4)    at (-0.50, 0);
\coordinate (q3)    at (-1.00, 0);
\coordinate (q2)    at (-1.40, 0.22);
\coordinate (q1)    at (-1.40, -0.22);

\coordinate (s1) at (1.60, 0);
\coordinate (s2) at (2.00, 0);
\coordinate (s3) at (2.40, 0);

\draw (z2)--(qm_2) (w2)--(qm_2);
\draw (z2)--(q4)   (w2)--(q4);
\draw (z2)--(q3)   (w2)--(q3);
\draw (z2)--(s1) (w2)--(s1);
\draw (z2)--(s2) (w2)--(s2);
\draw (z2)--(s3) (w2)--(s3);

\draw (q1)--(q2) (q1)--(q3);
\draw (q1)--(q4) (q2)--(q3);
\draw (q2)--(q4) (q3)--(q4);
\draw (q4)--(qm_2) (q4)--(qm_1);
\draw (q4)--(qm) (qm_2)--(qm_1);
\draw (qm_2)--(qm) (qm_1)--(qm);
\draw (s1)--(s2) (s2)--(s3);
\draw (s1) to[out=315, in=210, looseness=0.85] (s3);

\draw (z2)--(qm_1);
\draw (w2)--(qm);
\draw (z2) to[out=300, in=45, looseness=0.5] (qm);
\draw (w2) to[out=60, in=315, looseness=0.5] (qm_1);

\draw (z2) to[out=310, in=50, looseness=0.9] (w2);

\draw (z2)--(q2);
\draw (w2)--(q1);

\draw (z2) to[out=195, in=145, looseness=1] (q1);
\draw (w2) to[out=165, in=210, looseness=1] (q2);

\foreach \p in {z2,w2,qm_1,qm,qm_2,q4,q3,q2,q1,s1,s2,s3}{
    \node[dot] at (\p) {};
}

\node[above=3pt] at (z2) {$z$}; \node[below=3pt] at (w2) {$w$}; 
\end{tikzpicture}
\caption{The 1-planar graph $K_2\vee (B_7 \cup K_3)$}
\label{fig:B7-embedding}
\end{figure}
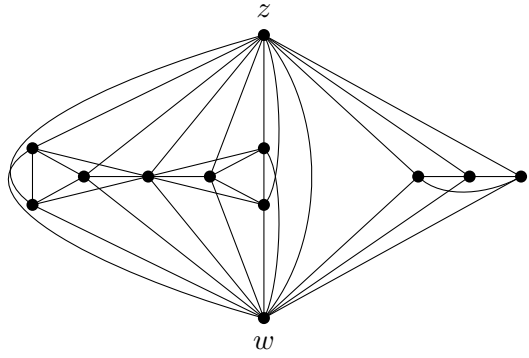

We now prove that $H$ necessarily contains a $P_4$.
\begin{fac}\label{fact:H-contains-P4}
$H$ admits a $P_4$.
\end{fac}
\begin{proof}[\bf Proof of Fact~\ref{fact:H-contains-P4}]
Suppose for contradiction that $H$ is $P_4$-free, so $H\cong T_{n-2}$ by Fact~\ref{fact:R-canonical}. We perform edge rewrites based on $n-2\pmod{3}$:
\begin{itemize}
\item If $n-2\equiv 0\pmod{3}$: $H=\tfrac{n-2}{3}K_3$; delete 9 edges from three disjoint triangles, add the 13 edges of $B_7\cup K_2$.
\item If $n-2\equiv 1\pmod{3}$: $H$ contains a unique $K_{1,6}$ component; delete 6 pendant edges of $K_{1,6}$ and add $E(B_7)$.
\item If $n-2\equiv 2\pmod{3}$: $H$ consists of $K_{1,4}$ and disjoint triangles; delete 7 edges of $K_{1,4} \cup K_3$ and add $E(B_7\cup K_1)$.
\end{itemize}
In all cases, the modified graph $H'$ satisfies the hypotheses of Fact~\ref{fact:edge-modify}: at most 30 edges are deleted, strictly more edges are added, and $K_2\vee H'$ remains 1-planar and $2C_5$-free. Thus $\rho(K_2\vee H')>\rho(G)$, contradicting extremality. We therefore deduce $H\supseteq P_4$.
\end{proof}

Let $Q$ denote the unique connected component of $H$ containing a $P_4$; all remaining components of $H-V(Q)$ are $P_4$-free, so Fact~\ref{fact:R-canonical} gives $H-V(Q)\cong T_{n-2-|V(Q)|}$. We next analyze the structural constraints on $Q$ imposed by the 1-planar embedding of $G=K_2\vee H$.

The join graph $K_2\vee H$ admits a 1-planar drawing such that vertices of each connected component of $H$ appear in consecutive cyclic order around vertex $z$. Among all such valid embeddings, we select the one minimizing the total number of edge crossings. Let $m = |V(Q)|$, and label vertices of $Q$ as $q_1,q_2,\dots,q_m$ consistent with their cyclic ordering around $z$. We introduce the key twin pair definition:
\begin{defi}
Two vertices $a,b\in V(H)$ form a \textit{twin pair} if some edge from $\{za, wa\}$ crosses an edge from $\{zb, wb\}$ in the fixed 1-planar drawing.
\end{defi}

\begin{fac}\label{fact:Q-twin-pair}
Assume $m=|V(Q)|\ge 7$. 
\begin{enumerate}[label=\textup{(\roman*)}]
\item If $Q$ contains no twin pairs, then $Q\subseteq P_m^2$;
\item If $Q$ contains exactly one twin pair, then after reversing vertex ordering if necessary, $Q\subseteq P_m^2 + q_1q_4 - q_3q_5$;
\item If $Q$ contains two twin pairs, then these are precisely $\{q_1,q_2\}$ and $\{q_{m-1},q_m\}$, and
$Q\subseteq P_m^2 + q_1q_4 + q_{m-3}q_m - q_3q_5 - q_{m-4}q_{m-2}.$
\end{enumerate}
\end{fac}
\begin{proof}[\bf Proof of Fact~\ref{fact:Q-twin-pair}]
Let $\{a,b\}\subseteq V(Q)$ be a twin pair, and suppose $zb$ crosses $wa$ at a point $o$. Then $a$ and $b$ are consecutive in the cyclic ordering of neighbors around $z$.
Indeed, suppose some vertex $u$ lies strictly between $a$ and $b$ in this ordering. Then $u$ lies in one of the two regions separated by the closed curves $zoaz$ and $wobw$. 
Without loss of generality, suppose $u$ lies inside the closed curve $zoaz$.
We now consider edge $uw$. 
Since edges $zb$ and $wa$ already cross at $o$, $uw$ must cross $za$, as depicted in Figure~\ref{fig:twin-row}(a) (the case where $uw$ does not wind around $z$; the other case is analogous). Hence no edge joins $\{a,b\}$ to $V(Q)\setminus\{a,b\}$, contradicting the fact that $Q$ is connected.
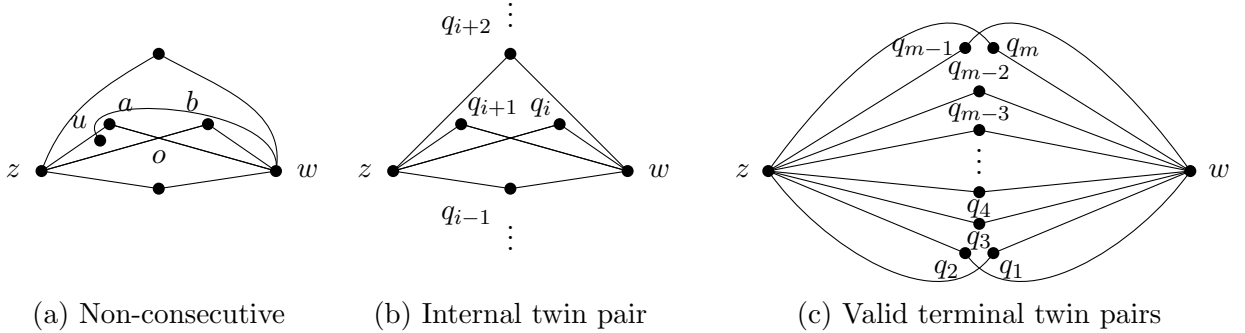
\begin{figure}[htbp]
\centering
\begin{tikzpicture}[
    dot/.style={circle,fill,inner sep=1.6pt},
    scale=1.55,
    every edge/.style={line width=1pt}
]

\begin{scope}[xshift=-3.0cm]
\coordinate (z)    at (-5.0, 0);
\coordinate (w)    at (-3.0, 0);
\coordinate (qim1) at (-4.0, -0.15);
\coordinate (b)    at (-3.58, 0.40);
\coordinate (a)    at (-4.42, 0.40);
\coordinate (qip2) at (-4.0, 1.0);
\coordinate (o)    at (-4.0, 0.12);
\coordinate (u)    at (-4.5, 0.26);

\draw (z)--(a)    (w)--(a);
\draw (z)--(b)    (w)--(b);
\draw (z)--(qim1) (w)--(qim1);
\draw (z)--(b);
\draw (w)--(a);
\draw (w) to[out=440, in=490, looseness=1] (u);
\draw (w) to[out=440, in=-30, looseness=1] (qip2);
\draw (z) to[out=70, in=210, looseness=1] (qip2);

\foreach \p in {z,w,qim1,b,a,qip2,u}{ \node[dot] at (\p) {}; }
\node[left=4pt]  at (z)    {$z$};
\node[right=4pt] at (w)    {$w$};
\node[above=1pt, xshift=-6pt]  at (b) {$b$};
\node[above=1pt, xshift=6pt]  at (a) {$a$};
\node at (o) {$o$};
\node[above left=1pt] at (u) {$u$};
\node[below=10pt] at (-4.0, -0.8) {(a) Non-consecutive};
\end{scope}

\begin{scope}[xshift=0cm]
\coordinate (z1) at (-5.0, 0);
\coordinate (w1) at (-3.0, 0);
\coordinate (qim1) at (-4.0, -0.15);
\coordinate (qi)   at (-3.58, 0.40);
\coordinate (qip1) at (-4.42, 0.40);
\coordinate (qip2) at (-4.0, 1.0);
\coordinate (opos) at (-4.0, 0.12);

\draw (z1)--(qip2) (w1)--(qip2);
\draw (z1)--(qip1) (w1)--(qip1);
\draw (z1)--(qi) (w1)--(qi);
\draw (z1)--(qim1) (w1)--(qim1);
\draw (z1)--(qi) (w1)--(qip1);

\node[font=\boldmath] at (-4.0,1.4) {$\vdots$};
\node[font=\boldmath] at (-4.0,-0.5) {$\vdots$};

\foreach \p in {z1,w1,qim1,qi,qip1,qip2}{ \node[dot] at (\p) {}; }
\node[left=4pt] at (z1) {$z$};
\node[right=4pt] at (w1) {$w$};
\node[below left=3pt] at (qim1) {$q_{i-1}$};
\node[above left=3pt] at (qip2) {$q_{i+2}$};
\node[above=-1pt, xshift=-7pt] at (qi) {$q_i$};
\node[above=-1pt, xshift=12pt] at (qip1) {$q_{i+1}$};
\node[below=10pt] at (-4.0, -0.8) {(b) Internal twin pair};
\end{scope}

\begin{scope}[xshift=0.0cm]
\coordinate (z2) at (-1.8, 0);
\coordinate (w2) at (1.8, 0);
\coordinate (qm_1) at (-0.12, 1.05);
\coordinate (qm)   at (0.12, 1.05);
\coordinate (qm_2) at (0, 0.68);
\coordinate (qm_3) at (0, 0.35);
\coordinate (q4)   at (0, -0.18);
\coordinate (q3)   at (0, -0.45);
\coordinate (q2) at (-0.12, -0.70);
\coordinate (q1) at (0.12, -0.70);

\node[font=\boldmath] at (0, 0.15) {$\vdots$};

\draw (z2)--(qm_2) (w2)--(qm_2);
\draw (z2)--(qm_3) (w2)--(qm_3);
\draw (z2)--(q4) (w2)--(q4);
\draw (z2)--(q3) (w2)--(q3);

\draw (z2) to[out=55,in=125,looseness=0.95] (qm);
\draw (w2) to[out=125,in=55,looseness=0.95] (qm_1);
\draw (z2)--(qm_1) (w2)--(qm);

\draw (z2) to[out=-55,in=-125,looseness=0.95] (q1);
\draw (w2) to[out=-125,in=-55,looseness=0.95] (q2);
\draw (z2)--(q2) (w2)--(q1);

\foreach \p in {z2,w2,qm_1,qm,qm_2,qm_3,q4,q3,q2,q1}{ \node[dot] at (\p) {}; }
\node[left=3pt]  at (z2) {$z$};
\node[right=3pt] at (w2) {$w$};
\node[left=0.1pt] at (qm_1) {$q_{m-1}$};
\node[right=1pt] at (qm)   {$q_m$};
\node[above=0.1pt] at (qm_2) {$q_{m-2}$};
\node[above=-1pt] at (qm_3) {$q_{m-3}$};
\node[below=-1pt] at (q4)   {$q_4$};
\node[below=-1pt] at (q3)   {$q_3$};
\node[left=7pt,below=-1pt] at (q2)   {$q_2$};
\node[right=7pt,below=-1pt] at (q1)   {$q_1$};
\node[below=10pt] at (0, -0.8) {(c) Valid terminal twin pairs};
\end{scope}

\end{tikzpicture}
\caption{Illustration of twin pairs in connected component $Q$.}
\label{fig:twin-row}
\end{figure}

Internal twin pairs $\{q_i,q_{i+1}\}$ with $2\le i\le m-2$ are forbidden; otherwise, set $A=\{q_1,\dots,q_{i-1}\}$ and $B=\{q_{i+2},\dots,q_m\}.$
Both $A$ and $B$ are nonempty. By definition of a twin pair, either $zq_i$ crosses $wq_{i+1}$, or $wq_i$ crosses $zq_{i+1}$; assume the former happens, as depicted in Figure~\ref{fig:twin-row}(b).
Since $zq_i$ crosses $wq_{i+1}$, we have $e_Q(A,B\cup\{q_{i},q_{i+1}\})=0$, contradicting the connectedness of $Q$.
Thus, only terminal twin pairs $\{q_1,q_2\},\{q_{m-1},q_m\}$ are permissible; see Figure~\ref{fig:twin-row}(c).

Suppose now that $\{q_1,q_2\}$ is a twin pair. By minimality of the total crossing number, the local configuration forces $q_1q_4,q_2q_4,q_3q_4\in E(Q)$.
Indeed, if $q_1q_4\notin E(Q)$, we can locally redraw the subgraph on $\{z,w,q_1,q_2,q_3,q_4\}$ to eliminate the crossing caused by the twin pair $\{q_1,q_2\}$, which reduces the total number of crossings.
If $q_1q_4\in E(Q)$ and $q_2q_4\notin E(Q)$, locally rearranging the vertices in the order $q_2,q_1,q_3,q_4$ reduces the total number of crossings.
Since $q_1q_4,q_2q_4\in E(Q)$, $q_3q_5\notin E(Q)$; otherwise, $q_3q_5$ would be crossed at least twice, violating $1$-planarity.
If $q_1q_4,q_2q_4\in E(Q)$ and $q_3q_4\notin E(Q)$, rearranging them as $q_3,q_1,q_2,q_4$ also reduces the total number of crossings.
By symmetry, if $\{q_{m-1},q_m\}$ is a twin pair, then $q_{m-3}q_{m-2},\ q_{m-3}q_{m-1},\ q_{m-3}q_m\in E(Q)$ and $q_{m-4}q_{m-2}\notin E(Q)$.
Apart from the two exceptional edges $q_1q_4$ and $q_{m-3}q_m$, every other edge $q_iq_j\in E(Q)$ satisfies $|i-j|\le 2$.
This proves (i)-(iii).
\end{proof}

\begin{fac}\label{fact:Q-size2}
$|V(Q)|\le 7$.
\end{fac}
\begin{proof}[\bf Proof of Fact~\ref{fact:Q-size2}]
Suppose for contradiction that $m=|V(Q)|\ge 8$. Let $k\in\{0,1,2\}$ denote the number of twin pairs in $Q$, and define the bounding supergraph $R_m^{(k)}$ as specified by Fact~\ref{fact:Q-twin-pair}.

Since $Q$ is connected, it contains a simple path $v_0v_1\ldots v_\ell$ with $v_0=q_1, v_\ell=q_m$ and $v_j=q_{i_j}$ for $0\leq j\leq \ell$, where $i_0=1$ and $i_{\ell}=m$. In $R_m^{(k)}$, the $k$ exceptional edges change the index difference by $3$, while every edge changes the index difference by at most $2$. Since the path is simple, each exceptional edge is used at most once. Hence $m-1\le \sum_{j=1}^{\ell}|i_j-i_{j-1}|\le 2\ell+k.$
Thus $\ell\ge \lceil \frac{m-1-k}{2}\rceil$. If $m\ge 14+k$, then $\ell\ge 7$, so the path contains a copy of $P_8$, contradicting Fact~\ref{fact:H-forbidden}. Therefore $m\le 13+k\le 15.$

For $h\ge 0$, define
$$
\eta(h):=
\begin{cases}
h, & h\equiv0\pmod3,\\
h-1, & h\equiv1,2\pmod3.
\end{cases}
$$
The function $\eta$ satisfies $\eta(a)+\eta(b)\le \eta(a+b)$ for all integers $a,b\ge 0$. 
Let $P$ be a copy of $P_4$ in $Q$, and let $p$ be the minimum possible value of $\max\{i:q_i\in V(P)\}$ over all such copies $P$. Thus $p\in\{4,\dots,m\}$. By Fact~\ref{fact:H-forbidden}, $Q^-:=Q[\{q_1,\ldots,q_{p-1}\}]$ and $Q^+:=Q[\{q_{p+1},\ldots,q_m\}]$ are $P_4$-free. By Fact~\ref{fact:P4-free-conn}, every $P_4$-free graph on $h$ vertices has at most $\eta(h)$ edges. Hence, $e(Q^-)\le \eta(p-1)$ and $e(Q^+)\le \eta(m-p)$.

Let $Z_p^{(k)}$ be the set of edges of $R_m^{(k)}$ which are contained in neither $R_m^{(k)}[\{q_1,\ldots,q_{p-1}\}]$ nor $R_m^{(k)}[\{q_{p+1},\ldots,q_m\}]$. Equivalently, $Z_p^{(k)}$ consists of edges incident to $q_p$ and edges between $V(Q^-)$ and $V(Q^+)$. By the structure of $R_m^{(k)}$, we have $|Z_p^{(k)}|\le5$, where $k\in\{0,1,2\}$. Every edge of $Q$ lies in $Q^-$, $Q^+$, or $Z_p^{(k)}$. Therefore 
$$
e(Q)\le \eta(p-1)+\eta(m-p)+5\le \eta(m-1)+5=\eta(m-7)+11.
$$

Let $F_m:=B_7\cup U_{m-7}$, where $U_{m-7}=\lfloor\frac{m-7}{3}\rfloor K_3\cup K_{m-7-3\lfloor\frac{m-7}{3}\rfloor}$ is an edge-maximal $P_4$-free 1-planar graph on $m-7$ vertices. Then $e(U_{m-7})=\eta(m-7)$, and so $e(F_m)=e(B_7)+\eta(m-7)\ge e(Q)+1.$
Setting $H'=(H-V(Q))\cup F_m$, the graph $K_2\vee H'$ is 1-planar and $2C_5$-free. Moreover, $e(Q)\le e(R_m^{(k)})\leq2m-2\le28<30.$ Thus $H'$ is obtained from $H$ by deleting at most $30$ edges and adding strictly more edges. By Fact~\ref{fact:edge-modify}, $\rho(K_2\vee H')>\rho(K_2\vee H)=\rho(G),$ contradicting the extremality of $G$. Therefore $m\ge 8$ is impossible, so $|V(Q)|\le 7$.
\end{proof}

\begin{fac}\label{fact:Q-B7}
$Q\cong B_7$.
\end{fac}
\begin{proof}[\bf Proof of Fact~\ref{fact:Q-B7}]
By Fact~\ref{fact:Q-size2} and the existence of $P_4\subseteq Q$, we have $4\le m\le7$.
We first exclude $4\le m\le6$. Since $Q$ is $C_5$-free, we have $e(Q)\le 6$ if $m=4$, $e(Q)\le 7$ if $m=5$ and $e(Q)\le 9$ if $m=6$, as illustrated in Figure~\ref{fig:extremal-C5-free}. 
\begin{figure}[h]
\centering
\begin{tikzpicture}[
    inner sep=0.5mm,
    place/.style={circle,draw=black,fill=black},
    r/.style={circle,draw=red,fill=red}]

\begin{scope}[xshift=0cm]
\node (p0) at (-0.7,-0.6)   [place] {};
\node (p1) at (-0.7,0.6)    [place] {};
\node (p2) at (0.1,  0)     [place] {};
\node (share) at (1.1,  0)  [place] {};

\draw[black, line width=0.7] 
(p0)--(p1) (p1)--(p2) (p2)--(p0)
(p0)--(share) (p1)--(share) (p2)--(share);

\node[below=4pt] at (0.2,-1.2) {$m=4$};
\end{scope}

\begin{scope}[xshift=3.6cm]
\node (p0) at (-0.7,-0.6)   [place] {};
\node (p1) at (-0.7,0.6)    [place] {};
\node (p2) at (0.1,  0)     [place] {};
\node (share) at (1.1,  0)  [place] {};
\node (leaf) at (1.9, 0)    [place] {};

\draw[black, line width=0.7] 
(p0)--(p1) (p1)--(p2) (p2)--(p0)
(p0)--(share) (p1)--(share) (p2)--(share)
(share)--(leaf);

\node[below=4pt] at (0.2,-1.2) {$m=5$};
\end{scope}

\begin{scope}[xshift=7.2cm]
\node (k0) at (-0.7,-0.6)   [place] {};
\node (k1) at (-0.7,0.6)    [place] {};
\node (k2) at (0.1,  0)     [place] {};
\node (share) at (1.1,  0)  [place] {};

\draw[black, line width=0.7] 
(k0)--(k1) (k1)--(k2) (k2)--(k0)
(k0)--(share) (k1)--(share) (k2)--(share);

\node (u1) at (1.8, 0.6)  [place] {};
\node (u2) at (1.8,-0.6)  [place] {};
\draw[black, line width=0.7] (share)--(u1) (share)--(u2) (u1)--(u2);

\node[below=4pt] at (0.3,-1.2) {$m=6$};
\end{scope}
\end{tikzpicture}
\caption{Edge-maximal $C_5$-free graphs of order $4,5,6$}
\label{fig:extremal-C5-free}
\end{figure}
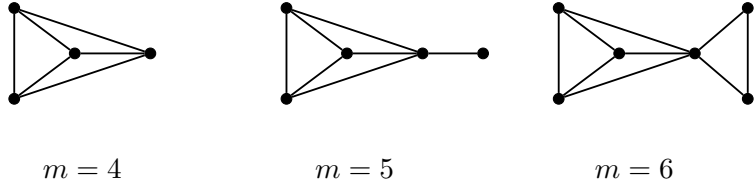
Note that $H-V(Q)$ contains a triangle $C_3$. Then one may easily see the upper bound on $e(Q\cup C_3).$ 
Replace $Q\cup C_3$, respectively, by $B_7$, $B_7\cup K_1$, and $B_7\cup K_2$, one may see $e(Q\cup C_3)< e(B_7), e(B_7\cup K_1), e(B_7\cup K_2)$. 
Furthermore, by Fact~\ref{fact:edge-modify} one sees $\rho(K_2\vee (Q\cup T_{n-6}))< \rho(K_2\vee (B_7\cup T_{n-9}) )$ for $m=4$,
$\rho(K_2\vee (Q\cup T_{n-7}))<\rho(K_2\vee (B_7\cup K_1 \cup T_{n-10}))$ for $m=5$, and
$\rho(K_2\vee (Q\cup T_{n-8}))< \rho(K_2\vee (B_7\cup K_2\cup T_{n-11}))$ for $m=6$, each of which contradicts the extremality of $G$. Therefore $m\notin\{4,5,6\}.$ Consequently, $m=7.$

It remains to show that $Q\cong B_7$. Let $k\in\{0,1,2\}$ be the number of twin pairs in $Q$. By Fact~\ref{fact:Q-twin-pair}, after possibly reversing the ordering $q_1,\ldots,q_7$, we have $Q\subseteq P_7^2$ if $k=0$, $Q\subseteq P_7^2+q_1q_4-q_3q_5$ if $k=1$ and $Q\subseteq P_7^2+q_1q_4+q_4q_7-q_3q_5$ if $k=2$.  
Since $e(P_7^2)=11,$ the cases $k=0$ and $k=1$ give $e(Q)\le11<12=e(B_7).$ Thus in either case replacing $Q$ by $B_7$ strictly increases the number of edges, and hence increases the spectral radius by Fact~\ref{fact:edge-modify}, a contradiction. Therefore, we consider the final subcase $k=2$. In this subcase $Q\subseteq P_7^2+q_1q_4+q_4q_7-q_3q_5 \cong B_7$. If $Q\ne B_7$, then $e(Q)\le11$, and replacing $Q$ by $B_7$ gives the same contradiction as above. Therefore $Q\cong B_7.$
\end{proof}

Combining Facts \ref{fact:R-canonical} and \ref{fact:Q-B7} gives us the global structural decomposition
$H\cong B_7\cup T_{n-9},$ which finally gives $G=K_2\vee (B_7\cup T_{n-9})$ for $t=2$. This completes the full proof of Theorem~\ref{t5.1}.
\end{proof}

\section{\normalsize Concluding remarks}

In this paper, we determine the spectral extremal graphs among $K_5$-free $1$-planar graphs and establish a
$K_{2,n-2}$-reduction theorem for the forbidden subgraphs considered in this work.
As an application, we characterize the spectral extremal 1-planar graphs for the $C_5$-free or $2C_5$-free cases. 
We believe that the structural approach developed here may also be useful for studying spectral Tur\'{a}n type problems involving other forbidden subgraphs in $1$-planar graphs. This motivates the following general problem.

\begin{pb}
For sufficiently large $n$, let $t\ge3$ be a fixed integer. Determine the spectral extremal graphs among $n$-vertex $1$-planar graphs without $t$ vertex-disjoint copies of $C_5$.
\end{pb}

For cycle lengths $\ell\ge6$, the interplay between the two dominating vertices and long cycles becomes more intricate. After the $K_{2,n-2}$-reduction, copies of $C_\ell$ can be generated by long paths and complex local substructures within $H$. This leads to our second open problem:

\begin{pb}
For sufficiently large $n$, let $t\ge1$ and $\ell\ge6$ be fixed integers. Classify the spectral extremal graphs among $n$-vertex $1$-planar graphs containing no $t$ vertex-disjoint copies of $C_\ell$.
\end{pb}

It would be interesting to identify the optimal component structures and build a unified framework to resolve the general $tC_\ell$-free spectral Tur\'{a}n type problem for all $\ell\ge5$.

Let $F$ be a given graph. We denote by $\operatorname{SPEX}(n,F)$ the class of graphs $G$ that attain the maximum adjacency spectral radius among all $n$-vertex $F$-free graphs. Nikiforov \cite{Nikiforov2007} proved that $\operatorname{SPEX}(n,K_{r+1}) \subseteq \operatorname{EX}(n,K_{r+1})$. Establishing the conjecture of Cioab\u{a} et al. (2022) \cite{CDT2022}, Wang et al. (2023) \cite{WKX2023} proved: for any graph $F$ such that the graphs in $\operatorname{EX}(n,F)$ are Tur\'an graphs plus $O(1)$ edges, $\operatorname{SPEX}(n,F)\subseteq \operatorname{EX}(n,F)$ for sufficiently large $n$. It is therefore natural to study the relationship between $\operatorname{SPEX}(n,F)$ and $\operatorname{EX}(n,F)$: $\operatorname{SPEX}(n,F) \subseteq \operatorname{EX}(n,F), \operatorname{SPEX}(n,F) \cap \operatorname{EX}(n,F)\not=\emptyset$, or $\operatorname{SPEX}(n,F) \cap \operatorname{EX}(n,F)=\emptyset.$ 

Zhang, Huang and Dong~\cite{Zhang2026edge} obtained $e(G)\le 4n-8$ for every $K_5$-free $1$-planar graph $G$ on $n\ge 3$ vertices.
Moreover, this bound is sharp for $n=8$ and for all integers $n\ge 10$. For the $K_5$-free $1$-planar spectral extremal graphs obtained in Theorem \ref{thm:XK5}, one sees $e(2K_1 \vee C_{n-2}^2)=4n-8,\, e(2K_1 \vee C_{n-2}^{2-})=4n-9.$
When $n$ is odd, the spectral extremal graph is not edge-extremal, implying that $\operatorname{SPEX}_{\mathcal{P}_1}(n,K_5) \not\subseteq \operatorname{EX}_{\mathcal{P}_1}(n,K_5)$.

Li~\cite{L2026} informed us that her group determined the maximum number of edges in $C_5$-free $1$-planar graphs on $n$ vertices is $3n-6.$ In view of Theorem~\ref{thm:C5}, $e(K_2\vee I_{n-2})=2n-2.$ Hence, $\operatorname{SPEX}_{\mathcal{P}_1}(n,C_5) \cap \operatorname{EX}_{\mathcal{P}_1}(n,C_5)=\emptyset$. It is an interesting task to figure out the essence of these mathematical phenomena.

\section*{\normalsize Disclosure statement}
The authors did not report any potential conflict of interest.
\section*{\normalsize Acknowledgement}
S.L. financially supported by the National Natural Science Foundation of China (Grant Nos.  12571365, 12171190),  and the Open Research Fund of Hubei Key Laboratory of Mathematical Science (Grant No. MPL2026ORG006).
\section*{\normalsize Data availability}
No data is available during the current study.

\end{document}